\documentclass[11pt]{article}

\usepackage{amsmath,amsthm,verbatim,amssymb,amsfonts,amscd,graphicx,soul}

\usepackage{xcolor}
\usepackage[utf8]{inputenc}
\usepackage{mathtools}

\usepackage{graphics}

\usepackage{euscript, enumerate}
\usepackage{url,hyperref}

\newcommand{\R}{{\mathbb R}}

\newcommand{\p}{\partial}
\newcommand{\fr}{\frac}
\newcommand{\la}{\langle}
\newcommand{\ra}{\rangle}
\newcommand{\na}{\nabla}

\newcommand{\e}{\epsilon}

\newcommand{\be}{\begin{equation}}
\newcommand{\ba}{\begin{aligned}}
\newcommand{\bee}{\begin{equation*}}
\newcommand{\ee}{\end{equation}}
\newcommand{\ea}{\end{aligned}}
\newcommand{\eee}{\end{equation*}}

\newcommand{\PP}{\bigl(\p_t-\fr{1}{H^2}\Delta\bigr)} 
\newcommand{\pp}{(\p_t-H^{-2}\Delta)} 
\newcommand{\n}{\nu}
\newcommand{\grad}{{\rm grad}\, }
\newcommand{\dist}{{\rm dist}\, }

\newcommand{\ds}{\displaystyle}
\newcommand{\La}{ {\Delta}  }

\theoremstyle{plain}
\newtheorem{theorem}{Theorem}[section]
\newtheorem{thm}[theorem]{Theorem}

\newtheorem{cor}[theorem]{Corollary}
\newtheorem{lemma}[theorem]{Lemma}
\newtheorem{lem}[theorem]{Lemma}

\newtheorem{prop}[theorem]{Proposition}

\newtheorem{claim}{Claim}[section]

\theoremstyle{remark}
\newtheorem{rem}{Remark}[section]
\newtheorem{remark}[rem]{Remark}
\newtheorem{example}[rem]{Example}

\theoremstyle{definition}
\newtheorem{definition}[theorem]{Definition}

\numberwithin{equation}{section}

\begin{document}
\title{Evolution of non-compact hypersurfaces by \\ inverse mean curvature}
\author{Beomjun Choi\thanks{{\color{black}B.~Choi  thanks the NSF for support in DMS-1600658}} and Panagiota Daskalopoulos\thanks{{\color{black}P.~Daskalopoulos thanks the NSF for support in DMS-1600658
and DMS-1900702.}}
}

\date{}
\maketitle

\begin{abstract}

We study the evolution of  complete non-compact convex hypersurfaces in $\R^{n+1}$ by the inverse mean curvature flow. We establish  the long time existence of solutions and provide the characterization of the maximal time of existence in terms of the tangent cone at infinity of the initial hypersurface. Our proof is based on an a'priori pointwise estimate on the mean curvature of the solution from below in terms of the aperture of a supporting cone at infinity.  The strict convexity of convex solutions is shown by means of viscosity solutions. Our  methods also give an alternative proof of the result by Huisken and Ilmanen in \cite{HI} on compact star-shaped solutions, based on  maximum principle argument. 

\end{abstract}


\section{Introduction}

A one-parameter family of i{\color{black}m}mersions $F : M^n \times [0,T]\to \R^{n+1}$  is a smooth complete solution to the {\em inverse mean curvature flow} (IMCF) in $\R^{n+1}$ if each $M_t:= F(\cdot,t)(M^n)$ is a smooth strictly mean convex complete hypersurface satisfying   
\be\label{eqn-IMCF} 
\frac{\partial}{\partial t} F(p,t)  =  H^{-1}(p,t)  \, \n(p,t)
\ee
where  $H(p,t)>0$  and $\nu(p,t)$ denote the mean curvature and outward unit normal of $M_t$, {\color{black}pointing opposite to the mean curvature vector.} 

\medskip
This  flow has been extensively studied for compact hypersurfaces. Gerhardt \cite{Ge2} and Urbas \cite{Ur} showed  compact smooth star-shaped strictly mean convex hypersurface admits  a  unique  smooth solution for all times $t \ge0$. Moreover,  the solution approaches to a homothetically  expanding sphere as $t\to\infty$.

For non-starshaped initial data it is well known that singularities may
develop (See  \cite{HI1} \cite{Sm}).  This happens when the  mean curvature  vanishes in some regions  which makes the classical
flow undefined. However, in  \cite{ HI1, HI2} Huisken and Ilmanen  developed a  level set approach to {\em weak
variational solutions}  of the flow which allows the solutions to {\em jump outwards}  in possible regions where $H=0$. Using the weak formulation, they gave the first proof of the {\em Riemannian Penrose inequality} in General Relativity.  One  key observation 
 in \cite{HI2} was the fact the Hawking mass of surface in  $3$-manifold of nonnegative scalar curvature is monotone under the weak flow, which was first discovered for classical solutions by Geroch \cite{Geroch}. Note that the Riemannian Penrose inequality was  shown  in more general settings by Bray \cite{Bray1} and Bray-Lee \cite{BL} by  different methods. 
Using similar techniques, the IMCF has been used to show geometric inequalities in various settings. For instance, see \cite{guan2007quermassintegral,BHW} for Minkowski type inequalities, \cite{lee2015penrose} for Penrose inequalities and \cite{makowski2013rigidity, de2016alexandrov,ge2014hyperbolic} for Alexandrov-Fenchel type inequalities among other results.  Note another important application of the flow by Bray and Neves in \cite{BN}. 

\medskip
In \cite{HI} Huisken and Ilmanen studied the  IMCF running from
compact  star-shaped weakly mean convex initial data.  
Using  star-shapedness and the ultra-fast diffusion character of the flow,   they derive a bound from above on $H^{-1}$ for  $t>0$ {\em which is independent of the initial curvature assumption}.  This follows by a  {\em Stampacchia iteration} argument and
utilizes  the {\em Michael-Simon Sobolev inequality}. The 
 $C^\infty$ regularity of solutions for $t>0$ easily  follows from the bound  on $H^{-1}$. The estimate in \cite{HI} is {local \em in time}, but necessarily {\em global in space} as it depends on the area of the initial hypersurface $M_0$ and uses global integration on $M_t$.
  As a consequence, the techniques in \cite{HI}  cannot be applied directly to  the non-compact setting.  Note that   \cite {li2017inverse} and \cite{zhou2017inverse} provide similar estimates for the IMCF in some negatively curved ambient spaces.  

\smallskip  
This work addresses {\em   the long time existence}  of {\em  non-compact}  smooth convex  solutions to the IMCF
embedded  in  {\em Euclidean space}
$\R^{n+1}$.  While extrinsic geometric flows have been extensively studied in the case of compact hypersurfaces, much remains to be investigated 
for  non-compact cases. The important  works by K. Ecker and G. Huisken \cite{EH1, EH}  address the evolution of entire graphs by {\em mean curvature flow} and  establish a surprising result:  existence for all times  with the only assumption that the initial data $M_0$ is a {\em locally Lipschitz} entire graph and  {\em no assumption of the growth at infinity of $M_0$.} This result  
is based on  priori estimates which are {\em  localized  in space}.  In addition,
the main local bound on the second fundamental form $|A|^2$ of $M_t$ is achieved without any  bound assumption on $|A|^2$  on $M_0$.  An  open question between  experts in the field has been  whether the techniques of Ecker and Huisken in
\cite{EH1, EH} can be extended to  the fully-nonlinear setting,  in particular on entire convex graphs 
evolving  by the  $\alpha$-{\em Gauss curvature flow} (powers $K^\alpha$ of the Gaussian curvature) and 
the {\em inverse mean curvature flow}. 

In \cite{CDKL} the second author,   jointly with Kyeongsu Choi, Lami Kim and Kiahm Lee,  established the long time existence of the  $\alpha$-{\em Gauss curvature flow} on any strictly  convex  complete non-compact hypersurface and for any $\alpha >0$. They showed  similar estimates 
as in  \cite{EH1, EH} which are {\em localized in space} can be obtained for this flow, however the method is more involved due to the {\em degenerate}  and
{\em fully-nonlinear character}  of the Monge-Amp{\color{black}\'ere} type of equation involved.   However, such localized results are not expected to hold
for the {\em  inverse mean curvature flow} where the {\em ultra-fast diffusion} tends to cause   instant  propagation from spatial infinity. 
In fact, one sees certain similarities between the latter two flows and the well known quasilinear models of diffusion on $\R^n$
\be\label{eq-ultra} 
u_t = \mbox{div} \big ( u^{m-1} \, \nabla u \big ).
\ee
Exponents $m >1$ correspond to degenerate diffusion while exponents $m <0$ to ultra-fast diffusion. We will see in the sequel that under the IMCF  the mean curvature $H$ satisfies an equation which is similar to \eqref{eq-ultra} with $m=-1$. 
Our goal is to study this phenomenon and establish the long time existence, an analogue of the results in \cite{EH1, EH} and \cite{CDKL}. 

{\color{black} Let us remark existing results on the IMCF of hypersurfaces other than closed ones. In \cite{A}, B. Allen investigated non-compact solutions in the hyperbolic space which are graphs on the horosphere.  One  key estimate  in \cite{A}    was to show that uniform upper and lower bounds on the mean curvature persist under some initial assumptions. The second author and Huisken \cite{DH} studied non-compact solutions in $\mathbb{R}^{n+1}$ under some initial conditions and we will discuss this result later this section. The flow with free boundary, i.e. solutions with Neumann-type boundary condition, has quite extensive literature. We refer the reader to \cite{St1,St2},\cite{Ed} and citations there-in for the mean curvature flow and \cite{M1,M2},\cite{LS} for the IMCF. }

\smallskip 

We will next state our main results. The following observation motivates the formulation of our theorem. 
\begin{example}[Conical solutions to IMCF]\label{example1}
For a solution $\Gamma_t$ to the IMCF   in $\mathbb{S}^n$, the {\em  family of cones  generated by  $\Gamma_t$} $$ \mathcal{C} \Gamma_t:= \{r x\in\mathbb{R}^{n+1} \,:\, r\ge0 ,\, x\in \Gamma_t \} $$  is a solution to the IMCF in $\mathbb{R}^{n+1}$ which is smooth except from  the origin.  
When $\Gamma_0$ is a compact smooth strictly convex, Gerhardt  \cite{Ge} and Makowski-Scheuer \cite{makowski2013rigidity} showed  the unique existence of solution for time $t\in [0, T)$ with $T<\infty$ and the convergence of solution  to an equator  as $t\to T$. Moreover, we have explicit formula $T=\ln |\mathbb{S}^{n-1}| - \ln |\Gamma_0|$ by the exponential growth of area in time, (2) in Lemma \ref{lem-HI1}.  {\color{black}Note also that $\mathcal{C}\Gamma_t$, restricted to the unit ball in $\mathbb{R}^{n+1}$, moves by the IMCF with free boundary on $\mathbb{S}^n$ in the sense of \cite{LS}.}
\end{example} 
From Example \ref{example1} and the ultra-fast diffusive character of the equation, it is reasonable to guess that the behavior of non-compact convex solution and its maximal time of existence is governed by the asymptotics at infinity.  For a non-compact convex set $\hat M_0$ and the associated hypersurface $M_0=\p \hat M_0$, we recall the definition of the {\em blow-down}, so called 
{\em the tangent cone at infinity}. 

\begin{figure}
\centering
\def\svgscale{1}{
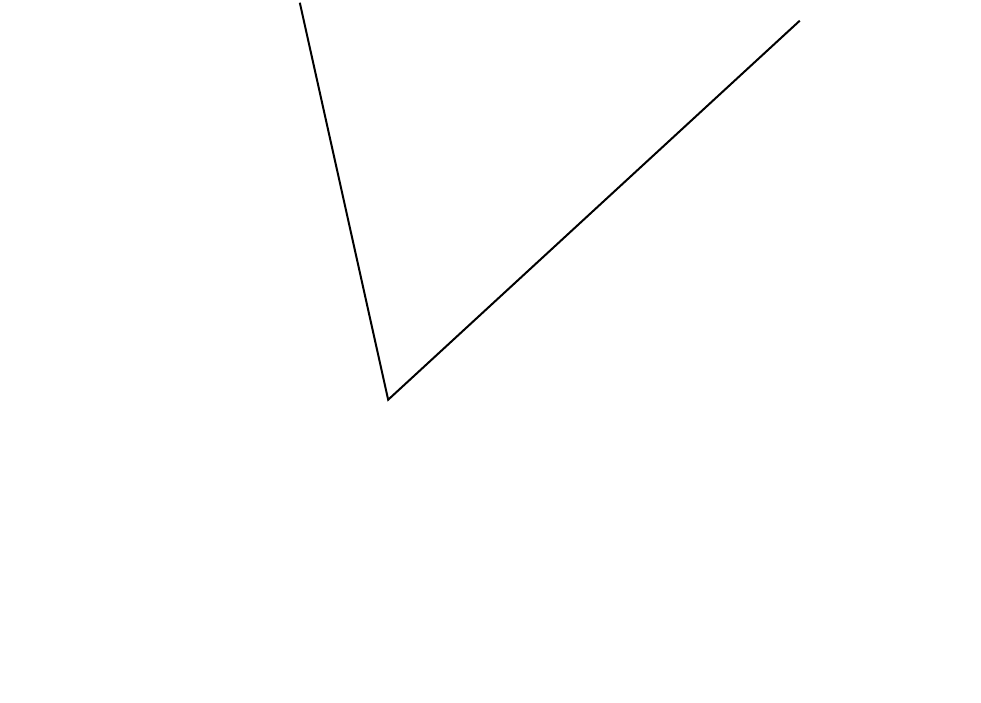
\caption{Definition \ref{defn-cone}}}
\end{figure}

\begin{definition}[Tangent cone at infinity]\label{defn-cone}{\em 
Let $\hat M_0\subset \mathbb{R}^{n+1}$ be a non-compact  closed convex set. For a point $p\in \hat M_0$, we denote  the tangent cone of $\hat M_0$ at infinity by $$\hat C_0:=\cap_{\lambda>0} \lambda(\hat M_0-p).$$ The definition is independent of $p\in\hat M_0$.  $C_0:= \p \hat C_0$ is called the tangent cone of $M_0=\p \hat M_0$ at infinity. $\hat \Gamma_0:= \hat C_0 \cap \mathbb{S}^{n}$ and $\Gamma_0:=C_0 \cap \mathbb{S}^n$ are called the links of tangent cones $\hat C_0$ and $C_0$, respectively.} 
\end{definition}

In this work, {\color{black}we say $M_0$ is convex hypersurface if it is}  the  boundary of a closed {\color{black} convex} set with non-empty interior. See Definition \ref{def-convexity} and {\color{black}subsequent} discussion for more details. For convex hypersurface $M_0$ in $\mathbb{R}^{n+1}$,  Lemma \ref{lem-elementary}
shows $M_0 = N_0 \times \R^k$ for some convex hypersurface $N_0$ in $\mathbb{R}^{n+1-k}$ which is homeomorphic to either $\mathbb{S}^{n-k}$   or  $\mathbb{R}^{n-k}$.  In the first case, the existence of compact IMCF, say $N_t$, running from $N_0$ is known in \cite{HI}  and thus $ N_t \times \R^k$ becomes a solution with initial data $M_0$. Therefore, the essential remaining  case is when $M_0$ is homeomorphic to $\mathbb{R}^n$.  We state our existence result.

\begin{theorem}\label{thm-mainexistence}
Let $\hat M_0$ in $\mathbb{R}^{n+1}$ with $n\ge2$ be a non-compact convex set with interior whose boundary $M_0$ is 
$C^{1,1}_{loc}$ and homeomorphic to $\mathbb{R}^n$, and $T=T(M_0)$ be a number defined by   \be \label {eq-T}T = \ln|\mathbb{S}^{n-1}| - \ln P(\hat\Gamma_0)\in[0,\infty].\ee Here, $\hat \Gamma_0$ is the link of tangent cone of $\hat M_0$ at infinity, $|\cdot | : =\mathcal{H}^{n-1}(\cdot)$, and $P(\hat \Gamma)$ is the perimeter of $\hat\Gamma$ in $\mathbb{S}^n$ defined by \be \label{eq-perimeter}P(\hat \Gamma_0) = \begin{cases}\begin{aligned} |\Gamma_0|\quad &\text{if }\hat \Gamma_0 \text{ has non-empty interior in }\mathbb{S}^n\\
2|\Gamma_0|\quad &\text{if }\hat \Gamma_0 \text{ has empty interior in }\mathbb{S}^n.\end{aligned}\end{cases}\ee 
{\color{black}If $T>0$, a smooth convex solution to the IMCF, say $M_t=\p\hat M_t$, exists for $0<t<T$.} $M_0$ is the initial data in the sense that  $M_t$ converges to $M_0$ locally uniformly as $t\to0$. $M_t$ is strictly convex if and only if $\hat M_0$ contains no infinite straight line inside.  \end{theorem}
%
%
%

\begin{remark}\label{remark-12} \quad

\begin{enumerate}[(i)] 
\item $\hat \Gamma_0$ can be an arbitrary convex set in $\mathbb{S}^n$ which may possibly have  empty interior. {\color{black}(See Definition \ref{def-convexity} and Lemma \ref{lem-convexitysphere} regarding the definition of convexity in $\mathbb{S}^n$.)} In that case the perimeter  $P(\hat \Gamma_0)$ is the  limit of outside areas of  decreasing sequence of convex sets with interior in $\mathbb{S}^n$ which approximate  $\hat \Gamma$. (See Lemma \ref{lem-b5} and Lemma \ref{lem-genouter}.)  According to \eqref{eq-T}, $T=\infty$ when $P(\hat \Gamma_0)=0$ and this happens if only if $\hat \Gamma$ has Hausdorff dimension less than $n-1$. Note that the definition of $P(\cdot)$ is not related with the notion of perimeter used in geometric measure theory.

\item The tangent cone of $M_t$ at infinity, say $\Gamma_t$, also evolves by IMCF in $\mathbb{S}^n$ in some generalized sense (Lemma \ref{lem-compgamma}), and becomes flat as $t\to T^-$ when $T<\infty$. In  Remark \ref{remark-blowdown} we further discuss this in connection with the {\em asymptotic behavior} of $M_t$ as $t\to T$.

\item  According to \eqref{eq-T}, $T=0$ when  $P(\hat\Gamma_0)=|\mathbb{S}^{n-1}|$.  In \cite{CH}, it was shown that   $P(\hat\Gamma_0)=|\mathbb{S}^{n-1}|$ if and only if $\hat \Gamma_0$ is either a hemisphere or a wedge \be\label{eq-wedge}\hat W_{\theta_0}=\mathbb{S}^{n}\cap\bigl(\{ (r\sin\theta, r\cos\theta)\,:\,  \theta\in [0,\theta_0], \text{ and }r>0 \}\times \mathbb{R}^{n-1}\bigr) \text{ for some }\theta_0\in[0,\pi)\ee  up to an isometry of $\mathbb{S}^{n}$.  {\color{black}We show in Theorem \ref{thm-noextension} no solution exists from such a $M_0$. }

\end{enumerate}
\end{remark}

\begin{remark} \label{rem-flatregion}Let us emphasize Theorem \ref{thm-mainexistence} allows  $H=0$ on a possibly non-compact region of $M_0$. Even in that case, $H$ becomes strictly positive for $t >0$ and this is due to Theorem \ref{thm-mainest3}. 
A similar phenomenon was observed for solutions to the Cauchy problem of the ultra-fast diffusion equation \eqref{eq-ultra} with $m<0$  on $\R^n$. See Remark \ref{remark-cauchy} for more  details.    

\end{remark}

%

Next result asserts that $T=T(M_0)$ in Theorem \ref{thm-mainexistence} is the maximal time of existence. The result holds not only for the solutions constructed in Theorem \ref{thm-mainexistence}, but applies to arbitrary solutions.
  \begin{thm}\label{thm-noextension} Let $M_0=\p \hat M_0$ satisfy  the same assumptions as in  Theorem \ref{thm-mainexistence} and $T=T(M_0)$ be given by \eqref{eq-T}. If $T<\infty$,  {\color{black} there is no smooth solution $N_t$ which is the boundary of $\hat N_t$, $\cap_{t>0} \hat N_t= \hat M_0$, and existing $0<t<T+\tau$ some $\tau>0$.} In particular, {\color{black}no solution exists if $T=0$.}  \end{thm} 
\medskip 
Non-compact IMCF in $\R^{n+1}$  was  first   considered by the second author and  G. Huisken in    \cite{DH},
where they  established 
the existence and uniqueness of smooth solution  to the IMCF, under the assumption that the initial 
 hypersurface $M_0$ is  an  entire $C^2$  graph, $x_{n+1}=u_0(x')$  with $H >0$,  in the following two cases:  \begin{enumerate}[(i)]
\item $M_0$ has {\em super linear growth}  at infinity and it is {\em strictly star-shaped},    that is
$  H \langle F - x_0, \nu \rangle \geq \delta >0$ 
 holds,  for some $ x_0 \in \R^{n+1}$; 
\item $M_0$ a {\em convex graph} satisfying  $0 < c_0 \leq H\,  \langle F - x_0, e_{n+1} \rangle \leq C_0 < +\infty$, for some $x_0 \in \R^{n+1}$  and lies   between {\em two round cones of the same aperture}, that is    \be\label{eqn-uuu}\alpha_0 |x'|\le u_0(x')\le \alpha_0 |x'| + k, \qquad \alpha_0 >0, \, k >0. 
\ee 
\end{enumerate}

In the first case, a unique smooth solution exists up to time $T=\infty$, while in the second case  a unique smooth convex solution  $M_t$ exists for $t\in[0,T)$ where $T>0$ is the time when the round conical solution from $\{x_{n+1} = \alpha_0 |x'|\}$ becomes flat. In the  latter case,  the  solution $M_t$ lies between  two evolving round cones and becomes flat as $t\to T$.  
To derive a local lower bound of $H$, a parabolic Moser's iteration argument was used along with a variant of Hardy's  inequality, which plays a similar role as the {Micheal-Simon Sobolev inequality} used in \cite{HI}.   

\smallskip 

Theorem \ref{thm-mainexistence} and the  results in \cite{DH} show that convex surfaces with {\em linear growth at infinity}  have 
{\em critical behavior}  in the sense that in this case the {\em  maximal  time of existence  is finite}  and it  depends on the {\em behavior  at infinity} of the initial data. However, while the techniques  in \cite{DH} only treat this critical linear case under the  condition \eqref{eqn-uuu},
Theorem \ref{thm-mainexistence} allows any behavior at infinity. Moreover,  the techniques in \cite{DH} require to assume that $H$ is
globally controlled from
below a{\color{black}t} initial time, namely that $  H \langle F - x_0, \nu \rangle \geq \delta >0$ in the case of {\em super-linear growth} and 
$H\,  \langle F - x_0, e_{n+1} \rangle \geq c >0$ in the case of {\em linear growth}. 

In this work we depart from the techniques in \cite{DH},\cite{HI}, and establish  a  priori bound on  $H^{-1}$ which is {\em local in time}.  For this, we develop  a new method  based on the  maximum principle rather than the integrations used in \cite{DH},\cite{HI}.   
Our key  estimate roughly says that a convex solution has a global bound on $(H|F|)^{-1}$ as long as a nontrivial convex cone is supporting the solution  from outside. 

 \begin{thm}\label{thm-mainest3}

Let $F:M^n\times [0,T] \to \R^{n+1}$, $n \geq 2$ and $T >0$,    be a compact smooth convex solution to the IMCF and suppose  there is  $\theta_1 \in (0,\pi/2)$ for which  
\be\label{eqn-theta1}
 \la F, e_{n+1}\ra \ge  {\sin\theta_1}\, |F|\quad  \text{ on }M^n\times[0,T].
 \ee
Then \be\label{eq-speedbd}\fr{1}{H|F|}\le C\left( 1+ \fr{1}{t^{1/2}}\right) \quad\text{on}\quad M^n\times [0,T] \ee for a constant $C=C(\theta_1)>0$.
\end{thm}

The compactness assumption on $M_t$ above  {\em will only be used to apply maximum principle} and will not 
affect the application of the estimate  in  proving  of our non-compact result,  Theorem \ref{thm-mainexistence}, as we will approximate non-compact solutions by compact ones.  Note the estimate does not depend on initial bound on $(H|F|)^{-1}$, which will allow initial data with flat regions as described in Remark \ref{rem-flatregion}.  Moreover, the new method developed while showing Theorem \ref{thm-mainest3} leads us to a new proof of  Theorem 1.1 in \cite{HI}, the $H^{-1}$ estimate  for compact, star-shaped solutions.  This is included in Theorem \ref{thm-equivHI} in the appendix. In fact, one expects  that similar estimates as in Theorem \ref{thm-equivHI}  can  be possibly derived   for the IMCF in other ambient spaces, including some positively curved spaces or asymptotically flat spaces, using this new method and this generalizes the results of \cite{HI,li2017inverse, zhou2017inverse}.   See in \cite{HI3} for a consequence of such an estimate when this is shown in asymptotically flat ambient spaces. 

\smallskip

\begin{remark}\label{rmk-14} Recently, the first author and P.-K. Hung  in \cite{CH}  addressed  the IMCF on convex solutions allowing singularities on $M_0$. Using Theorem \ref{thm-mainest3} as  a key ingredient, \cite{CH} shows the tangent cone obtained after blowing-up at a singularity point evolves by the IMCF. As a corollary, one can  consider an arbitrary non-compact convex hypersurface $M_0$ in Theorem \ref{thm-mainexistence} and obtain the following  necessary and sufficient condition for the existence of a smooth solution: {\em for an arbitrary non-compact convex $M_0$ with $T(M_0)>0$, there is a smooth solution if and only if $M_0$ has density one  everywhere. i.e. $\Theta_0(p) = \lim_{r\to0} \fr{|B_r(p) \cap M_0 |}{\omega_n r^n} =1$  for all $p\in M_0$.}  See \cite{CH} for more details. 
\end{remark}

\medskip
{\em A   brief outline   of this paper}  is as follows:   In Section \ref{sec-Preliminaries}, we introduce basic  notation, evolution equations of basic geometric quantities, and prove some useful identities.  
 Section \ref{sec-main-non com2} is devoted to the proof of  main a priori estimate Theorem \ref{thm-mainest3}.  In Section \ref{sec-existence}, we prove the long time existence of solution (Theorem \ref{thm-mainexistence} and Theorem \ref{thm-noextension}) via an approximation argument. Here, the passage to a limit relies on the estimate Theorem \ref{thm-mainest3}. In Appendix \ref{sec-strict}, we prove the convexity of solution is preserved and show the solution become{\color{black}s} strictly convex immediately unless the lowest principle curvature $\lambda_1$ is zero everywhere. This will be shown for the solutions to the IMCF in space forms as this adds no difficulty in the proof but could be useful in other application. In Appendix \ref{appendix-starshaped}, we give an alternative proof of $H^{-1}$ estimate shown in \cite{HI} using a maximum principle argument, showing how the star-shapedness condition can be incorporated in our method. 
Finally,  in Appendix \ref{appendix-approximation}  we show the approximation theorems of convex hypersurfaces in $\mathbb{R}^{n+1}$ and $\mathbb{S}^{n}$ that are used throughout the paper. 

\medskip 
{\color{black} \noindent{\bf Acknowledgement: }  The authors wish to express their gratitude to Gerhard Huisken and Pei-Ken Hung for 
stimulating and useful discussions on the inverse mean curvature flow.} 

\section{Preliminaries}\label{sec-Preliminaries}

In this section we present some  basic preliminary results. Let us  begin by clarifying some notions and simple facts  from convex geometry. 
Convex hypersurfaces  are studied  in both convex geometry and differential geometry of submanifolds. As a result  there are different notions of convexity preferred in different subjects. In this paper, we use the following definition. 

\begin{definition}[Convex hypersurfaces in $\mathbb{R}^{n+1}$ and $\mathbb{S}^{n}$] \label{def-convexity} \quad

\begin{enumerate}[{ (i)}] 
\item  A hypersurface  $M_0 \subset  \mathbb{R}^{n+1}$ is called convex   if it is the boundary of a convex set $\hat M_0$ of  non-empty interior and  $\hat M_0 \neq \mathbb{R}^{n+1}$.
\item A set $\hat \Gamma_0\subset \mathbb{S}^n$ is  convex if for all $p$ and $q$ in $\hat\Gamma_0$ at least one minimal geodesic connecting the two points is contained in $\hat \Gamma_0$. $\Gamma_0$ is a convex hypersurface in $\mathbb{S}^n$ if it is the boundary of a convex set {\color{black}$\hat\Gamma _0$} with non-empty interior and $\hat \Gamma_0\neq \mathbb{S}^{n}$.
\item A  $C^2$ convex hypersurface is strictly convex at a given point if the second fundamental form with respect to the inner  the normal is positive definite. 
\end{enumerate}
\end{definition}

There is a useful characterization of convexity in $\mathbb{S}^n$ which is immediate from Definition \ref{def-convexity}.

\begin{lemma}\label{lem-convexitysphere}
 A set $\hat\Gamma_0\subset $ $\mathbb{S}^n$ is convex if and only {\color{black}if} it is connected and 
$ \mathcal{C} \hat \Gamma_0 := \{r x\in \R^{n+1}:  r\ge0, \,   x\in {\color{black}\hat\Gamma_0}\}$ 
 is  convex in $\mathbb{R}^{n+1}$.
\end{lemma} 
 We prefer Definition \ref{def-convexity} to the other one that  defines convexity  through certain properties of the embedding or immersion since we will deal with convex hypersurfaces of low regularity.  These two notions are, however, equivalent under suitable assumptions. For example, if $M_0\subset \mathbb{R}^{n+1}$ (or $\Gamma_0 \subset\mathbb{S}^n$) is a convex hypersurface, then it is a complete connected embedded submanifold. Furthermore, if  $M_0$ (or $\Gamma_0$) is   $C^2$, then the second fundamental form with respect to the inner normal is nonnegative definite.  The converse question, namely under what conditions  an immersed or embedded hypersurface of nonnegative sectional curvature (or semi-definite second fundamental form) bounds  a convex set, has a long history and has been answered, for instance, by Hadamard \cite{Hadamard}, Sacksteder-van Heijenoort \cite{Sacksteder, Heijenoort}, H. Wu \cite{Wu}, do Carmo-Warner \cite{CW}, Makowski-Scheu{\color{black}er} \cite{makowski2013rigidity} under different assumptions. We refer {\color{black} the reader} to the results and references cited in these papers.  
 
\smallskip

The following simple observation will be used throughout this paper.

\begin{lemma}\label{lem-elementary} Let $M_0=\p \hat M_0$ be the boundary of a closed convex set with interior $\hat M_0 \subset \mathbb{R}^{n+1}$, that is  $M_0$ is a convex hypersurface in $\mathbb{R}^{n+1}$. Then either $M_0=\mathbb{R}^n$ or $M_0= \mathbb{R}^k \times N_0$,  for some $0\le k < n$ and $N_0=\p \hat N_0$ where $\hat N_0\subset \mathbb{R}^{n+1-k}$ is a closed convex set with interior which contains no infinite line. Moreover, such $N_0$ is either homeomorphic to $\mathbb{S}^{n-k}$  or $\mathbb{R}^{n-k}$. 

\begin{proof}
If   $\hat M_0$ contains an infinite straight line, then  $\hat M_0$  splits off in the direction of the line by the following elementary argument: suppose a line $\{t e_{n+1} \in\mathbb{R}^{n+1} \,:\, t\in\mathbb{R} \}$ is contained in $\hat M_0$ and let us denote its cross-sections $\hat \Omega_l \times\{l\} := \hat M_0 \cap \{x_{n+1} =l\}$ for $\l\in\mathbb{R}$. By convexity, for any $0\le t_1<t_2$  the set $\frac{t_2-t_1}{t_2} \hat \Omega_0 \times\{t_1\}$ is contained in $\hat M_0\cap \{x_{n+1} =t_1\}$. By taking $t_2\to \infty$ while fixing $t_1$, we have $\hat \Omega_{0}\times \{t_1\} \subset \hat M_0 \cap \{x_{n+1}=t_1\}$. We can do a similar argument for $t_2<t_1 \le 0$ and thus $\hat \Omega_0 \times \mathbb{R} \subset \hat M_0$. Similarly, we can do the   same argument  for all other sections $\hat \Omega_{l}\times\{l\}$  and obtain that $\Omega_{l}\times \mathbb{R} \subset \hat M_0$. Therefore $\hat \Omega_{l_1}=\hat \Omega_{l_2}$ for all $l_1\neq l_2$ and $\hat M_0 = \hat \Omega_0 \times \mathbb{R}$. By repeating this splitting,  we conclude that $\hat M_0 = \hat N_0 \times \R^k$,   for some $k\ge 0$ where  $\hat N_0$ does not contain any infinite lines. In this case, a  classical 
simple result  in convex geometry (see for instance Lemma 1 in \cite{Wu}), implies that $\p \hat N_0$ is homeomorphic to either ${\color{black}\mathbb{S}}^{n-k}$ or $\mathbb{R}^{n-k}$.

\end{proof}
\end{lemma}

\smallskip

We next  derive some properties of smooth solutions to IMCF.  Let $\nabla := \nabla^{g(t)}$ and $\Delta := \Delta_{g(t)}$  denote the connection and Laplacian on $M^n$ with respect to  the induced metric $g_{ij}(t)=\langle \fr{\partial F}{\partial x^i},\fr{\partial F}{\partial x^j} \rangle $.  Recall that on a local system of coordinates  $\{ x^i\}$ on $M^n$,  
\be\label{eqn-FF}\fr {\partial^2 F }{\partial x^i \partial x^j }= -h_{ij}\n+ \Gamma_{ij}^k \fr{\partial F}{\partial x^k}\qquad\text{ and}\qquad\left \langle \fr{\partial F}{\partial x^j}, \fr{\partial \n}{\partial x^i}\right\rangle = h_{ij}
\ee
where $\nu$ denotes the outer unit normal. 
We also define the operator $$\square := \left(\partial_t - \fr{1}{H^2}\Delta\right)$$ and use it  frequently as this is the linearized operator of the IMCF.  The IMCF or generally curvature flows of homogeneous degree $-1$, have the following scaling property which we will  frequently use:
 \begin{lemma}[Scaling of IMCF]\label{lem-scaling} If $M^n_t\subset \mathbb{R}^{n+1}$ is a solution to the IMCF, then $\tilde M^n_t= \lambda \,  M^n_t$ is again a solution for $\lambda>0$. \end{lemma} 
\begin{lemma}[Huisken, Ilmanen \cite{HI}] \label{lem-HI1} Any smooth  solution  of the IMCF \eqref{eqn-IMCF} in $\R^{n+1}$ satisfies 
 \begin{enumerate}[{\em (1)}]
\item $  \partial_t  g_{ij} =  \frac 2H \, h_{ij}$
\item $  \partial_t d\mu = d\mu$,  \, where $d\mu$ is the volume form induced from $g_{ij}$
\item $  \partial_t \n =  - \nabla H^{-1} =   \frac 1{H^2} \, \nabla H$
\item $ \PP h_{ij} = -  \frac 2{H^3} \nabla_i H \nabla_j H + \frac{ |A|^2}{H^2}  \, h_{ij}$
\item $  \partial_t  H  =  \nabla_i \big  ( \frac 1{H^2}  \, \nabla_i H \big ) -     \frac{|A|^2}{H} = 
\frac 1{H^2} \Delta H  -  \frac 2{H^3} |\nabla H|^2 -     \frac{|A|^2}{H}$
\item $  \PP  H^{-1}=      \frac{|A|^2}{H^2}\, H^{-1}$
\item $ \PP  \langle F -  x_0, \n \rangle =  \frac  { |A|^2}{H^2}  \, \langle F -  x_0, \n \rangle$. 
\end{enumerate}
\end{lemma}
\begin{remark}\label{remark-diffambient} If the ambient space is not $\R^{n+1}$, then the evolution equations of $g_{ij}$, $d\mu$, and $\nu$ remain the same as in $\R^{n+1}$, but the evolution of curvature $h_{ij}$  is  different and complicated. On a space form of sectional curvature $K$, the formula  hugely simplifies becoming  
\be\label{eq-noneuch_ij}\p_t h_{ij}=\fr{1}{H^2}\Delta h_{ij}+\fr{|A|^2}{H^2} h_{ij} - \fr{2}{H^3}\nabla_i H\nabla _jH -\fr{nKh_{ij}}{H^2}\ee
(See Chapter 2 in \cite{gerhardt2006curvature}.) 
In this paper we will mostly focus on the flow  in Euclidean space and we will only use \eqref{eq-noneuch_ij}  in Appendix \ref{sec-strict}. 
\end{remark}

Using Lemma \ref{lem-HI1} one can easily deduce the following formulas. 

\begin{lem} \label{lem-23}For a fixed vector $\omega$ in $\R^{n+1}$, the smooth solutions to the IMCF  \eqref{eqn-IMCF} in $\R^{n+1}$ satisfy  
 \begin{enumerate}[{\em (1)}]
\item $  \PP |F- x_0|^2=  -\frac{2n}{H^2} + \fr{4}{H}\langle F- x_0,\n\rangle$
\item $   \PP  \langle \omega, \n \rangle =\fr{|A|^2}{H^2}  \langle \omega, \n \rangle  $
\item $  \PP \langle \omega, F-x_0 \rangle =\fr{2}{H}  \langle \omega, \n \rangle  $.

\end{enumerate}
\begin{proof} By \eqref{eqn-FF} we have  \[\Delta F = g^{ij} (\p^2_{ij}F - \Gamma_{ij}^k F_k)=g^{ij} (-h_{ij} \n + \Gamma_{ij}^kF_k - \Gamma_{ij}^k F_k) = -H \, \n.\] 
which combined with $\p_tF= H^{-1}\n$ implies (3).  Next, \[\Delta | F-x_0|^2 = 2 \la \Delta F , F-x_0\ra + 2 \la \nabla F, \nabla F\ra ={\color{black}-} 2H\la \n, F-x_0\ra +2n \] implies (1). Finally, 
\[\ba\Delta \n &= g^{ij}(\p^2_{ij}\n-\Gamma_{ij}^k \p_k \n ) = g^{ij}(\p_j (h_{i}^k F_k)- \Gamma_{ij}^k h_{k}^l F_l)\\&=g^{ij}((\p_j h_{i}^k)F_k - h_{i}^kh_{jk}\n + \Gamma_{jk}^l h^k_iF_l- \Gamma_{ij}^k h_{k}^l F_l)\\
&=-|A|^2\n+g^{ij}\nabla_j h^k_iF_k= -|A|^2 \n + \nabla H\ea\]  
where we used the Codazzi identity  in the last equation. This implies (2).
\end{proof}
\end{lem}

The following simple lemma, which commonly appears in Pogorelov type computations,  will be   useful in the sequel when we compute the evolution of products. 
\begin{lemma} \label{lem-pg}
For any $C^2$ functions $f_i(p,t)$, $i=1,\ldots,m$, denote  \[w:= f_1^{\alpha_1}f_2^{\alpha_2}\ldots f_m^{\alpha_m}.\] 
Then  on the region where  $w \neq 0$, we have   
\be\label{eqn-product}\PP \ln |w|=\fr{\pp w}{w} +\fr{1}{H^2}\fr{|\nabla w|^2}{ w^2}=\sum_{i=1}^m\alpha_i\left( \fr{\pp f_i}{f_i}  +\fr{1}{H^2} \fr{|\nabla f_i|^2}{f_i^2}\right). 
\ee 
\begin{proof}
The lemma simply follows from \be \label{eq-pg}\PP \ln |f| =  \fr{\pp f}{f}  +\fr{1}{H^2} \fr{|\nabla f|^2}{f^2}.\ee 
\end{proof}
\end{lemma}

Next two lemmas are straightforward computations and we leave their proofs for readers. 
\begin{lem} \label{lem-com2}For any  two $C^2$ functions $f$, $g$ defined on $M^n\times(0,T)$ and any $C^2$ function $\psi : \R \to \R$, 
$$\square (fg) =(\square f )g + f(\square g )- \fr{2}{H^2}\langle \nabla f,\nabla g \rangle$$
and $$\square \psi(f) = \psi'(f)\square {\color{black}f}-\fr{\psi''(f)}{H^2}|\nabla f|^2$$
where  $\square:=\pp$.
\end{lem}
\begin{lem}\label{lem-com1} If  a $C^2$ function $f$ is defined on a solution $M_t$ of the IMCF and satisfies  
$$  \left(\frac{\partial}{\partial t} - \fr{1}{H^2} \La\right) f= \fr{|A|^2}{H^2} f$$ then for any fixed $\beta \neq0$ we have

$$   \left(\frac{\partial}{\partial t} - \fr{1}{H^2} \La\right) f^{\beta} = \beta\fr{|A|^2}{H^2} f^{\beta}  - \fr{\beta(\beta-1)}{\beta^2} \fr{|\nabla f^{\beta}|^2}{H^2f^{\beta}}.$$
For instance, $H^{-1}$, $\langle \omega, \n \rangle$ and $\langle F-x_0, \n\rangle $ are  examples  of such a  function $f$.
\end{lem}
     
We finish with the following   local estimate  which is an easy consequence of Proposition 2.11 in \cite{DH}. Here $B_r$ denotes  an extrinsic ball of radius $r>0$ in $\mathbb{R}^{n+1}$. 
\begin{prop}[Proposition 2.11 \cite{DH}]\label{lem-DH}  For a solution $M_t$, $t \in [0,T]$,  to the IMCF,  there is a constant $C_n>0$ such that\[\sup _{M_t \cap B_{r/2} } H \le C_n \max \,( \sup _{M_0 \cap B_{r} } H,\, r^{-1}).\] 
\begin{proof}  Although this proposition is proven in \cite{DH} we include below its proof for completeness. 
For fixed $r>0$, let $\eta:= (r^2- |F|^2)_+^2$ be a cut-off function defined in the ambient space. Using Lemma \ref{lem-23} and \ref{lem-com2},
\[  \ba \left(\frac{\partial}{\partial t} - \fr{1}{H^2} \La\right)  \eta &=-2(r^2-|F|^2)_+ \left[  \left(\frac{\partial}{\partial t} - \fr{1}{H^2} \La\right) |F|^2 \right]-\frac{2}{H^2}  |\nabla (r^2-|F|^2)_+|^2 \\
&= -2 \eta^{1/2} \left(-\frac{2n}{H^2} + \fr{4}{H}\langle F,\n\rangle\right) -\frac{8}{H^2} |F^T|^2\ea \]
and
\[ \ba \left(\frac{\partial}{\partial t} - \fr{1}{H^2} \La\right) \eta H &= \frac{4n \eta^{1/2}}{H} -{8\eta^{1/2}\la F ,\nu\ra}- \frac{8}{H}|F^T|^2+ \eta \left (-2\fr{|\nabla H|^2}{H^3} - \frac{|A|^2}{H} \right)-\frac{2}{H^2} \la \nabla \eta ,\nabla H\ra  \\
&\le \frac{4n \eta^{1/2}}{H} +{8\eta^{1/2}r}-\frac{2}{H^{\color{black}3}} \la \nabla (\eta H), \nabla H\ra- \frac{\eta H}{n}.\ea\]
In the last inequality, we used $|A|^2 \ge n^{-1} H^2$ and $|\la F,\nu \ra|\le |F| \le r$. Let $m(t)$ be the maximum of $\eta H $ on $M_t$. Then the above inequality implies 
\[\p_t m(t) \le 4n \frac{\Vert \eta\Vert _{\infty}^{3/2}}{m(t)} + 8 \Vert \eta\Vert _{\infty}^{1/2}r - \frac{m(t)}{n}\le 4n \frac{r^6}{m(t)} + 8r^3 -\frac1n m(t).\]
Thus $m(t)$ will decrease if 
$$ r^6 \, \frac {4n } {m(t)}  -  \frac  {m(t)} n+ 8r^3\leq 0  \iff m^2(t)-8nr^3 m(t) -4n^2r^6 \geq 
0  \iff m(t) \geq  (4+2\sqrt{5})n \, r^3.
$$
Therefore, 
$
m(t) \leq \max \, (\, m(0),(4+2\sqrt{5})n\, r^3\, ). 
$
The proposition is implied since
\[\ba \sup _{M_t \cap B_{r/2} } (r^2 -(r/2)^2)^2 H &\le \sup \eta H \le  \max \, (\, m(0), (4+2\sqrt{5})n\, r^3\, ) \\&\le  \max \, (\, \sup _{M_0 \cap B_{r}  }r^4H , (4+2\sqrt{5})n\, r^3\, ) .\ea \]

\end{proof}
\end{prop} 
\section{$L^\infty$ bound of  $(H|F|)^{-1}$} \label{sec-main-non com2}

Before giving the proof of Theorem \ref{thm-mainest3}, lets us introduce some notations. We consider  spherical coordinates with respect to the origin in $\R^{n+1}$,
namely  \[x=(x_1,\ldots,x_{n+1}) = ( r \omega\,  \sin \theta, r\cos \theta)\quad \text{ with }r\ge0,\, \, \omega \in \mathbb{S}^{n-1},\,\, \mbox{and }\, \theta \in[0,\pi]\] 
which are  smoothly well-defined {\color{black}away} from $x_{n+1}$-axis. We will also denote by $\bar \nabla $ and $\nabla$  metric-induced connections on $(\R^{n+1}, g_{\text{euc}})$  and $(M^n, F^* g_{\text{euc}})$, respectively.   Before the proof, we need the evolution equation of $\theta$, defined in the ambient space as follows: 

\begin{definition} We define 
\be\label{eqn-theta10}\theta: \R^{n+1}\setminus\{0\} \rightarrow [0,\pi]\quad \text{ by } \,\, \theta(x):= \arccos \left(\fr{\la x,e_{n+1}\ra}{|x|} \right)
\ee
and 
\[r:\R^{n+1}\rightarrow [0,\infty) \quad \text{ by }r(x):=|x|.\] Moreover, we define smooth unit orthogonal vector fields \[e_\theta(x) =e_\theta(x',x_{n+1}):= \fr{1}{|x|}\fr{\p \,}{\p \theta}= \left(\fr{x'}{|x|}\fr{\cos\theta}{\sin\theta},-\sin \theta\right) \quad \text{ on }\R^{n+1}\setminus \{x_{n+1}-\text{axis}\}\]  and \[e_r(x):=\fr{\p\,}{\p r} =\fr{x}{|x|} \quad  \text{ on }\R^{n+1} \setminus \{0\}.\]
Though $\theta$ is not smooth at the points on the $x_{n+1}$-axis, note that  $\theta^2$, $\cos\theta$, and $\sec\theta$ are all smooth on $\{ x_{n+1} >0\}$.
\end{definition} 

{\color{black}Note $\theta$ represents a scaled distance from the north pole measured in the sphere. The first negative term on the right hand side of \eqref{eq-thetaevol} will justify the use of $\theta$ in the estimate. Compare \eqref{eq-thetaevol} with (1) in Lemma \ref{lem-23}.}
\begin{lem}\label{lem-theta}On the region $\{\theta\neq 0,\pi\}\cap \{|x|\neq0\}$, 
\be\label{eq-thetaevol}\ba \PP \theta
& = -\fr{1}{H^2r^2}\left(\fr{n-|\nabla r|^2}{\tan \theta}\right)+\fr{1}{H^2}\fr{|\nabla \theta|^2}{\tan \theta } +\fr{2}{H^2} \la \fr{\nabla r}{r}, \nabla \theta \ra + \fr{2}{H} \la \nu, \bar \nabla \theta\ra.\ea\ee 

\begin{proof}
Consider a spherical coordinate chart $$(r,\theta,(w^\alpha)_{\alpha=1\ldots n-1} )\quad\text{with }r>0,\ \theta\in(0,\pi),\ (w^\alpha) \in \mathbb{S}^{n-1}   $$ 
around a point $\{\theta \neq 0,\pi\}\cap\{|x|\neq 0\}$ in $\R^{n+1}$, where $w^\alpha$ is a coordinate chart of $\mathbb{S}^{n-1}$. 
On this chart, \be \label{eq-metric}g_{\text{euc}} =dr^2 + r^2\, d\theta^2 + r^2\sin^2\theta \sigma_{\alpha\beta} \, dw^\alpha   dw^\beta .\ee 
Also note that  \be\label{eq-grad}\grad \theta = \fr{1}{r^2} \fr{\p}{\p \theta}=\fr{1}{r}e_{\theta}\quad\text{and}\quad \grad r = \fr{\p}{\p r}=e_r\quad \text{on }\, (\R^{n+1}, g_{\text{euc}}) .\ee 

At a given $p\in M^n$ with $\{\theta\neq 0,\pi\}\cap \{|x|\neq0\}$, let us choose a geodesic normal coordinate of $M^n$, say $\{ y^i\}_{i=1}^n$. In this coordinate at this point, 
$$\ba \Delta \theta &={\color{black} \sum_{i}}\p_{i}\p_i \theta= \sum_{i}\fr{\p}{\p y^i}d\theta \left(\fr{\p}{\p y^i}\right)=\sum_{i}\fr{\p}{\p y^i}\left(\fr{1}{r^2} \la  \fr{\p}{\p\theta}, \fr{\p}{\p y^i}\ra  \right) \\
&=\sum_{i} -\fr{2}{r^2} \la\fr{\p}{\p y^i}, e_\theta\ra \la\fr{\p}{\p y^i}, e_r \ra +\fr{1}{r^2} \la \bar\nabla_{\p_{i}} \fr{\p}{\p \theta}, \fr{\p}{\p y^i} \ra+ {\color{black}\frac1{r^2}} \la \fr{\p}{\p \theta},-h_{ii}\nu \ra .\ea$$
Since $\left( \left(\fr{\p\,}{\p y^i}\right)_{i=1}^n,\nu\right)$ constitutes an orthonormal basis of $T_{F(p)}\R^{n+1}$, 
\be \label{eq-innerp}\sum_{i}\la\fr{\p}{\p y^i}, e_\theta\ra \la\fr{\p}{\p y^i}, e_r \ra +  \la \nu, e_r\ra \la \nu, e_\theta\ra= \la e_r, e_\theta \ra=0 .\ee
Therefore, \be\label{eqn-latheta}
\Delta \theta = -\fr{H}{r} \la \nu, e_\theta\ra + \fr{2}{r^2} \la \nu, e_r\ra \la \nu. e_\theta \ra   +\fr{1}{r^2}\sum_i \la \bar\nabla_{\p_{i}} \fr{\p}{\p \theta}, \fr{\p}{\p y^i} \ra.
\ee

\begin{claim}\label{claim-111} 
\be\label{eqn-sum100}\ba\sum_{i=1}^n\la \bar\nabla_{\p_{i}} \fr{\p}{\p \theta}, \fr{\p}{\p y^i}\ra&= \fr{\cos\theta}{\sin\theta}\left(n-(1- \la \nu, e_r\ra^2 )- (1-\la \nu,e_\theta\ra^2) \right).\ea
\ee
\begin{proof}[Proof of Claim \ref{claim-111} ]
By computing the Christoffel symbols from the metric \eqref{eq-metric}, we get: \be\label{eq-symbols}\bar\nabla_{\fr{\p}{\p\theta}} \fr{\p}{\p\theta} = -r \fr{\p}{\p r},\quad \bar\nabla_{\fr{\p}{\p r}} \fr{\p}{\p\theta} =\fr{1}{r}\fr{\p}{\p\theta},\quad \bar\nabla_{\fr{\p\,\,}{\p w^\alpha}} \fr{\p}{\p\theta} = \fr{\cos\theta}{\sin\theta} \fr{\p\,\,}{\p w^\alpha}.\ee
Suppose $\p_i=\p_{y^i}= a_{\theta} \p_\theta + a_{r} \p_r + \sum_\alpha a_{\alpha} \p_{w^\alpha}$. Then ${\color{black}\bar \na}_{\p_i} \fr{\p}{\p \theta} =-r a_\theta \p_r + \fr{a_r}{r} \p_\theta + \sum_\alpha a_\alpha\fr{\cos \theta}{\sin\theta} \p_{w^\alpha}$ and hence 
\[\ba \la {\color{black}\bar\na}_{\p_i} \fr{\p}{\p \theta}, \fr{\p}{\p y^i} \ra&= -ra_\theta a_r+ ra_\theta a_r + r^2\sin^2\theta \fr{\cos\theta}{\sin\theta} a_\alpha a_\beta \sigma^{\alpha\beta}\\
&=\fr{\cos\theta}{\sin\theta}\left[\left|\fr{\p}{\p y^i}\right|^2 - \left\la\fr{\p}{\p y^i}, e_r \right\ra^2 - \left\la\fr{\p}{\p y^i}, e_\theta\right\ra^2 \right] .\ea\]
The claim follows by summing this over $i$.\end{proof}
\end{claim} 

\noindent Now ${\ds \p_t\theta=d\theta(\p_tF) = \fr{1}{H} \la \nu, \grad \theta \ra = \fr{\la\nu, e_\theta \ra }{rH}}$, \eqref{eqn-latheta}
and \eqref{eqn-sum100}  imply    
\[\PP \theta= \fr{2\la \nu ,e_\theta\ra}{rH}- \fr{1}{(rH)^2} \left[\fr{\cos\theta}{\sin\theta}[n-(1- \left\la \nu, e_r\right\ra^2) -(1- \left\la \nu, e_\theta\right\ra^2)] + 2 \la\nu, e_r \ra \la \nu, e_\theta\ra \right].\] 
Hence, the  lemma follows  by using  \eqref{eq-grad} and the orthonormality of $\left( \left(\fr{\p\,}{\p y^i}\right)_{i=1}^n,\nu\right)$ in the
equation above.  
\end{proof}
\end{lem} 

\smallskip 

\begin{proof}[Proof of Theorem \ref{thm-mainest3}] Using the definition \eqref{eqn-theta10},  our condition \eqref{eqn-theta1} can be  written as $\theta(p,t) \le \pi/2 -\theta_1$. 
Setting  ${\ds c := \fr{\pi-\theta_1}{\pi-2\theta_1}>1}$,  we have   $c \, \theta \le  \fr{\pi}{2}-\fr{\theta_1}{2} <\fr{\pi}{2}$ and $\sec (c \theta)\le 2\sec \theta$ for $\theta =\theta(p,t)$ on $t\in[0,T]$. 

\medskip\noindent By lemma \ref{lem-com2},\[\ba\square \sec(c \theta) &= c  \sec(c \theta) \tan (c \theta) \square \theta -\fr{1}{H^2}c ^2 [\sec(c \theta)\tan^2(c \theta) +\sec^3(c \theta) ] |\nabla \theta|^2] \\
&=\sec(c\theta)\left [ c\tan(c\theta) \square \theta - \fr{c^2}{H^2} (2\tan^2(c\theta)+1) |\na\theta|^2\right].\ea\] After defining  $\varphi:=\sec(c \theta)$, Lemma \ref{lem-theta} and $\fr{\na \varphi}{\varphi} = c\tan(c\theta) \na\theta$ imply  \[\ba\fr{\square\varphi}{\varphi} &= -\fr{c }{H^2r^2} \fr{\tan(c \theta)}{\tan \theta} (n-|\nabla r|^2-r^2|\nabla\theta|^2)  + \fr{2}{H^2}\bigl \la \fr{\nabla r}{r}, \fr{\nabla \varphi}{\varphi} \bigr\ra+\fr{2}{H} \bigl\la \nu, \fr{\bar\nabla \varphi}{\varphi}\bigr\ra\\
&\quad-\fr{2}{H^2}\fr{|\nabla \varphi|^2}{\varphi^2} -\fr{1}{H^2}c ^2 |\nabla\theta |^2 \\
 &\mbox{(since }n-|\nabla r|^2 -r^2|\nabla \theta|^2 =n-2\ge 0\quad\mbox{and}\quad \fr{\tan(c \theta)}{\tan\theta} \ge c )\\
&\le  -\fr{c ^2}{H^2r^2} (n-|\nabla r|^2) -\fr{2}{H^2} \fr{|\nabla \varphi|^2}{\varphi^2} +\fr{2}{H^2} \bigl\la \fr{\nabla r}{r}, \fr{\nabla \varphi}{\varphi} \bigr\ra+\fr{2}{H} \bigl\la \nu, \fr{\bar \nabla \varphi}{\varphi}\bigr\ra. \ea\] Note that this inequality  holds on $\{ x_{n+1} >0\}$, where our solution $M_t$ is located. 
Let ${\ds w:=\fr{\sec (c \theta)t}{Hr}= \varphi\psi r^{-1}t}$ where $\psi :=H^{-1}$. 
Then by Lemma \ref{lem-pg} and the previous inequality \be \label{eq-we}\ba&\fr{\square w}{w}+ \fr{1}{H^2}\fr{|\nabla w|^2}{w^2} \\
=& \bigl(\fr{\square \varphi}{\varphi}+\fr{1}{H^2} \fr{|\nabla \varphi|^2}{\varphi^2}\bigr) +\bigl( \fr{|A|^2}{H^2}+\fr{1}{H^2}\fr{|\nabla\psi|^2}{\psi^2} \bigr) -\fr{1}{2} \bigl( \fr{\square r^2}{r^2} + \fr{1}{H^2}\fr{|\nabla r^2|^2}{r^4}\bigr)+\fr{1}{t}\\
=& \bigl(\fr{\square \varphi}{\varphi}+\fr{1}{H^2} \fr{|\nabla \varphi|^2}{\varphi^2}\bigr) +\bigl( \fr{|A|^2}{H^2}+\fr{1}{H^2}\fr{|\nabla\psi|^2}{\psi^2} \bigr) +\bigl( \fr{n}{H^2r^2} -\fr{2}{H} \la \nu,  \fr{\bar \na r}{r} \ra- \fr{2}{H^2}\fr{|\nabla r|^2 }{r^2}\bigr)+\fr{1}{t}\\
\le& \left[\fr{|A|^2}{H^2} +\fr{1}{t} +\fr{2}{H} \bigl\la \nu, \fr{\bar \nabla \varphi}{\varphi}\bigr\ra-\fr{2}{H}\bigl\la \nu, \fr{\bar\nabla r}{r} \bigr \ra\right]\\
&\quad+\fr{1}{H^2} \left[ \fr{|\nabla \psi|^2}{\psi^2} + \fr{n}{r^2}-2\fr{|\nabla r|^2}{r^2} -c^2\fr{n-|\nabla r|^2}{r^2}-\fr{|\nabla \varphi|^2}{\varphi^2} +2\bigl \la\fr{\nabla r}{r}, \fr{\nabla \varphi}{\varphi}\bigr\ra \right] \\=&:(1)+(2).\ea\ee
Suppose a nonzero maximum of $w(p,t)$ on $M^n\times [0,t_1]$ is achieved at $(p_0,t_0)$ with $t_0\in(0,t_1]$. At this point, \[0=\fr{\nabla w}{w}= \fr{\nabla \psi}{\psi} + \fr{\nabla \varphi}{\varphi} - \fr{\nabla r}{r}\]and therefore \[\fr{|\nabla \psi|^2}{\psi^2} = \left | \fr{\nabla r}{r}-\fr{\nabla \varphi}{\varphi}\right|^2= \fr{|\nabla r|^2}{r^2} + \fr{|\nabla \varphi|^2}{\varphi^2} -2 \bigl \la \fr{\nabla r}{r}, \fr{\nabla \varphi}{\varphi} \bigr\ra. \]
At the maximum point, by plugging this into $(2)$ in \eqref{eq-we}, ${\ds (2)= -(c^2-1)  \fr{n-|\nabla r|^2}{H^2r^2}}$. 
Therefore at the maximum point,  \[0\le (1)  -(c^2-1) \fr{n-|\nabla r|^2}{H^2r^2}  .\]
Let us estimate terms in $(1)$. Note that by our choice of $c>1$, \[\left|\fr{\bar\nabla \varphi}{\varphi}\right| =| c\tan (c\theta) \bar\nabla \theta |\le \fr{c}{r} \tan(c\theta)=\fr{1}{r} \sin(c\theta)\sec(c\theta) \le \fr{2}{r\cos \theta}\le \fr{2}{r} \fr{1}{\sin \theta_1} =\fr{C}{r}\] for some $C=C(\theta_1)$. Next,  $\fr{|A|^2}{H^2} \le 1$ from convexity and  $|\bar\nabla r| \le 1$ imply at $(p_0,t_0)$, 
\[\ba0&\le -(n-|\na r|^2)\fr{c-1}{H^2r^2}+\fr{C}{Hr}+1+\fr{1}{t_0}\\
 &\le -\fr{c-1}{H^2r^2}+\fr{C}{Hr}+1+\fr{1}{t_0} \quad (\text{since }|\na r|^2\le 1,\, n\ge2)\\
 &\le -\fr{c-1}{2H^2r^2} + \fr{C}{2(c-1)} +\fr{1}{t_0}.\ea\]
Note that $0<t_0\le t_1$ and $1\le \varphi \le C$ on $M\times [0,t_1]$. 
Multiplication of  $(\varphi(p_0,t_0)t_0)^2$ implies \[w^2(p_0,t_0)=\left(\fr{\varphi t_0}{Hr}\right)^2 \le C t_1  \, (t_1+1).\] On any other point  $p\in M$ at $t=t_1$,\[\fr{1}{H^2r^2}(p,t_1) = \left(\fr{w(p,t_1)}{t_1\varphi(p,t_1)}\right)^2 \le \fr{w^2(p_0,t_0)}{t_1^2\varphi^2(p,t_1)} \le C\left(1+\fr{1}{t_1}\right).\] We used $\varphi\ge1$ in the last inequality. This finishes the proof of Theorem \ref{thm-mainest3}.
\end{proof}
\begin{remark}\label{rem-initial} If we define $\bar w:= \varphi\psi r^{-1}$ and follow the rest similarly, we get an estimate which includes the initial bound\[\fr{1}{H|F| } \le C \max \left( \sup_{M_0} \fr{1}{H|F|} , 1\right).\]
\end{remark} 
{\color{black}After the preprint of this work has been posted, M.N. Ivaki pointed that $\varphi (\sec \theta)$ with $\varphi$ of form \eqref{eqn-functions} could replace the use of $\sec (c\theta)$ in the previous proof. If $\varphi (\sec \theta)$ is used, one may avoid computing the evolution of $\theta$ as the evolution of  $\sec \theta = \la F, e_{n+1} \ra / |F|$ can be derived from Lemma \ref{lem-23}. }

  \section{Long time existence of non-compact solutions}\label{sec-existence}


Let us provide a sketch on how we prove Theorem \ref{thm-mainexistence}. Since Theorem \ref{thm-mainest3} was shown for compact solutions, we first construct a family of {\em compact convex} approximating solutions $M_{i,t}=\p \hat M_{i,t}$ which is monotone increasing in $i$. Each compact expanding solution $M_{i,t}$ exists for all time by \cite{Ge2}\cite{Ur}, thus we may define the limit $\hat M_{t} = \lim_{i \to +\infty} \hat M_{i,t}$ as a set. 
We will see, however, that the limit $\hat M_t$ is non-trivial only up to time $t<T$ and the proof of Theorem \ref{thm-noextension} will show $\hat M_t=\R^{n+1}$  for $t>T$, meaning $M_t=\p \hat M_t$ is empty.

Let us briefly explain where the connection between our solution $M_t$ in Euclidean space and solutions on the sphere is revealed. Recall $\hat\Gamma_0 $ is the link of the tangent cone of $\hat M_0$. For each $\delta>0$, there is a smooth strictly convex $\Gamma_0^\delta \subset\mathbb{S}^n$ such that  $ \hat \Gamma_0 \subset\subset \hat \Gamma_0^\delta$ and $T_\delta:=\ln |\mathbb{S}^n|-\ln|\Gamma_0^\delta| >T-\delta$. Here and later, we use $\hat A \subset \subset \hat B$ to denote $\overline{\hat A}\subset \text{int}(B)$. As explained in   Example \ref{example1}, $\Gamma_0^\delta$ admits a smooth  $\Gamma^\delta_t $ in $\mathbb{S}^n$ which exists up to time $T_\delta$ 
 and we will make use of $\mathcal{C}\Gamma_t^\delta$ as a barrier which contains $M_{i,t}$. Indeed, by moving its vertex far away from $M_0$ initially, we can make 
 $\mathcal{C}\Gamma_t^\delta$ (after the initial translation)   contain  $M_{i,t}$ up to time $T_\delta$, implying  that each $M_{i,t}$ satisfies condition \eqref{eqn-theta1} in Theorem \ref{thm-mainest3}   up to time $T-\delta$ for all $\delta>0$. Theorem \ref{thm-mainest3}  then leads to an upper  bound on $(|F|\, H)^{-1}$ showing that the IMCF on $M_{i,t}$ is  locally uniformly  parabolic and the rest is straightforward.  
 
\smallskip 
 The proof of Theorem \ref{thm-mainexistence} consists of four steps. First, we define our solution $M_t = \p\hat M_t$ as a limit 
of compact approximating solutions. Second, we show that $M_t$ is a smooth non-trivial solution for $t\in(0,T)$ using the  idea mentioned above. Third, we show that $M_t$ locally converges to $M_0$, as $t\to 0$. Finally,  we show  the strict convexity assertion of Theorem \ref{thm-mainexistence}. Since the proof is long, we will address some proofs of technical lemmas in Appendix.

\begin{proof}[Proof of Theorem \ref{thm-mainexistence}]  {\color{black}Suppose $T=T(M_0)$ satisfies $0<T\le \infty$ following the assumption of theorem.} We may also assume without loss of generality that $\hat M_0$ {\em does not contain any infinite straight lines}. Let us justify this claim. By Lemma \ref{lem-elementary}, $M_0= \mathbb{R}^k \times N_0$ for some non-compact convex hypersurface $N_0\subset \mathbb{R}^{n+1-k}$ and it is homeomorphic to $\R^{n-k}$ as $M_0$ is homeomorphic to $\R^n$. 
Note $0\le k\le n-2$ since $k=n-1$ or $n$ {\color{black}would imply  $\hat \Gamma_{0,\hat M_0}$ is a wedge in \eqref{eq-wedge} (when $k=n-1$) or a hemisphere (when $k=n$), and $T=0$ in both cases.} If $\hat \Gamma_{0,\hat M_0} \subset \mathbb{S}^{n}$ and $\hat \Gamma_{0, \hat N_0} \subset \mathbb{S}^{n-k}$ are the links of the tangent cones of $\hat M_0$ and $\hat N_0$, then we have  \be\label{eq-reduction} T(M_0)=\ln {|\mathbb{S}^{n{\color{black}-1}}|}-\ln {P(\hat \Gamma_{0,\hat M_0})} = \ln {|\mathbb{S}^{n-k{\color{black}-1}}|}-\ln{P(\hat \Gamma_{0,\hat N_0})}=T(N_0). \ee

In conclusion, $N_0^{n-k}=\p \hat N_0$ is a non-compact convex hypersurface in $\mathbb{R}^{(n-k)+1}$ (with $n-k\ge2$) homeomorphic to $\mathbb{R}^{n-k}$,
 i.e. $N_0$ satisfies the assumption of Theorem \ref{thm-mainexistence} and $\hat N_0$ has no infinite straight line inside. Since the existence of a solution $N_t$  of IMCF in $\R^{n-k+1}$ with initial data  $N_0$ implies that $M_t := \mathbb{R}^{k} \times N_t$ is  a solution of IMCF in $\R^{n+1}$ with initial data  $M_0$,
 we conclude that it would be sufficient, as we claimed, to assume that   $\hat M_0$ does not contain any infinite straight lines. In  this case,  
 the {\em link $\hat \Gamma_0$}  of the tangent cone of $\hat M_0$  {\em  does not contain any antipodal points}   and is {\em compactly contained in some open hemisphere}
 $H(v_0):=\{ p\in \mathbb{S}^n\, : \, \la p, v_0 \ra >0\}$ for some $v_0\in \mathbb{S}^n$ (see   Lemma 3.8 in \cite{makowski2013rigidity}).   

\smallskip 

Next, we create hypersurfaces $M_{i,0}=\p\hat M_{i,0}$ with certain properties which approximate $M_0$ from inside. A sequence of sets $ \hat A_i $ is called {\color{black}strictly} increasing if $\overline{\hat A_i}\subset \text{int}(\hat A_{i+1})$ (and we  write $A_i\subset\subset A_{i+1}$).  Let us denote the ball of radius $r$ centered at $p$ by $B_r(p)$. Setting $\hat \Sigma_{i,0}:=B_i(0)\cap \hat M_0$, $\hat \Sigma_{i,0} $  is a {\color{black}(weakly)} increasing sequence of  convex sets. By Lemma \ref{lem-euclidapprox}, each $\Sigma_{i,0}$ admits a strictly increasing approximation by compact smooth strictly convex hypersurfaces $\Sigma_{i,j}=\p \hat \Sigma_{i,j}$.  Furthermore, we may assume $d_H(\hat \Sigma_{i,0}, \hat \Sigma_{i,j}) \le j^{-1}$,  where $d_H$ is the Hausdorff metric \eqref{eq-Hausdorff}. Then a diagonal argument gives a sequence  $n_i \to \infty$ so that $\hat M_{i,0}:=\hat \Sigma_{i,n_i}$ 
strictly increases to $\hat M_0$. By \cite{Ur} and \cite{Ge},  each $M_{i,0}$  admits a unique smooth IMCF  $M_{i,t}=\p \hat M_{i,t}$,  which exists for all  $t\in[0,\infty)$. 
Note $M_{i,t}$ is smooth strictly convex hypersurface (see Remark \ref{remark-conv}) which is strictly monotone increasing in $i$ by the comparison principle.  {\color{black}We use the sequence of  solutions $M_{i,t}$ to define next the notion of {\em innermost candidate} solution from $\hat M_0$.  \begin{figure}
\centering
\def\svgscale{1}{
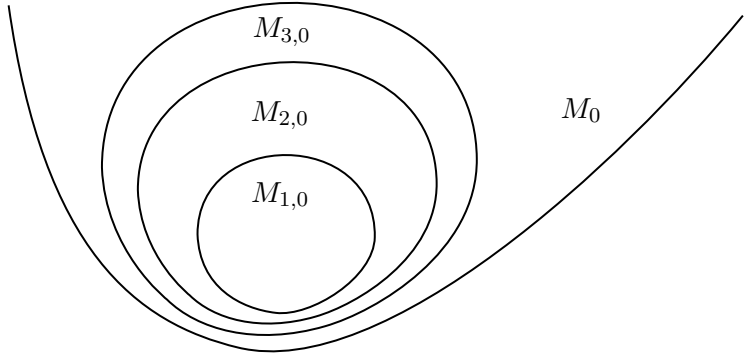{
\caption{Approximation of $M_0$}}}
\end{figure}

\begin{definition}\label{def-candidate} For a convex set $\hat M_0$ with non-empty  interior, let $\hat M_{i,0}$  be a sequence of compact sets with
smooth strictly convex boundary which strictly increases   to $\hat M_{0}$. 
We  define the {\em innermost candidate} solution from $\hat M_0$ as  \be\label{eq-candidate} \quad \hat M_t := \overline{\cup_{i=1}^{\infty} \hat M_{i,t}},  \quad \text{for }t\in[0,\infty)\ee 
where  $M_{i,t}=\p\hat M_{i,t}$  is a sequence of compact smooth strictly convex solutions  with initial data $M_{i,0}$.  $\hat M_t$ is convex by definition and the definition does not depend on $\hat M_{i,t}$ (See Remark \ref{remark-candidate}.)
\end{definition} }It remains to show  that  {\em $M_t :=\p \hat M_t$ defines  a non-empty strictly convex smooth solution to the flow, for $t\in(0,T(M_0))$, and converges to $M_0$ locally uniformly as $t\to 0^+$}. We need  an approximation lemma  for $\hat \Gamma_0$, the link   of the tangent cone   of $\hat M_0$ at infinity.

 \begin{claim}\label{claim-approximation}
Let $\hat\Gamma_0 \subset \mathbb{S}^n$ be a closed convex set contained in an open hemisphere.  Then there is a 
sequence of smooth, strictly convex hypersurfaces $\Gamma^j_0=\p\hat \Gamma^j_0$ in the open hemisphere which strictly decreases and $\cap_{j} \hat \Gamma_0^j = \hat \Gamma_0$. For every such sequence, $|\Gamma_0^j|=P(\hat\Gamma^j_0)\to P(\hat \Gamma_0).$

\begin{proof}[Proof of Claim \ref{claim-approximation}] This is a direct consequence of Lemma \ref{lem-outersphere} and Lemma \ref{lem-b5}. Since their   proofs 
require some other results from  convex geometry, we prove  them separately in the appendix.
\end{proof}
\end{claim}

\smallskip 
 Now fix an arbitrary time   $t_0\in(0,T)$. 
By the claim, we may find a $j$ such that $T^{j}:=\ln |\mathbb{S}^{n-1}| - \ln | \Gamma^{j}_0| > t_0$. Since $\hat \Gamma_0$ is contained in the interior of $\hat \Gamma^{j}_0$, we may find a vector $v'_{j}\in \mathbb{R}^{n+1}$ so that $ \hat M_0 \subset \mathcal{C} \hat \Gamma_0^{j}+v_{j}'.$ 
By the definition of $\hat M_0$, it then follows that $\hat M_{i,0} \subset \hat M_0 \subset \mathcal{C} \hat \Gamma_0^{j}+v_{j}'$, for all $i$. 
Theorem 1.4 \cite{makowski2013rigidity} guarantees the existence of  a  smooth strictly convex IMCF solution  $\Gamma_t^{j}$  in $\mathbb{S}^n$ with initial data  $\Gamma^{j}_0$, for $t\in[0,T^{j})$. {\color{black}Note $\mathcal{C} \Gamma_t^j$ is an IMCF which is smooth away from the origin. Unless the convex cone $\mathcal{C} \Gamma_t^j$ is flat, the origin can not be touched from inside by a $C^2$ hypersurface. Therefore a version of comparison principle can be justified and we obtain $\hat M_{i,t} \subset \mathcal{C} \hat \Gamma^{j}_t+v_{j}' $ for $t\in[0,T^{j})$.} Since $\Gamma^{j}_t$ is a strictly convex solution which converges to an equator, we may find a direction $\omega_0 \in \mathbb{S}^n$ and small $\delta_0>0$ such that \[\la F-v_j' , \omega_0 \ra \ge ( \sin \delta_0 )\, |F-v_j'|,  \qquad\text{ for } t\in [0,t_0] \,\, \text{ on } \,\, M_{i,t} .  \]
By Theorem \ref{thm-mainest3}, we have a uniform bound \be\label{eq-HFv} (H|F-v'_j|)^{-1} \le C(1+t^{-1/2}) \quad \text{for }M_{i{\color{black},}t}\text{ on }t\in(0,t_0].\ee 
The  barrier $ \mathcal{C} \hat \Gamma^{j}_t+v_j' $ also shows $\hat M_t\neq \mathbb{R}^{n+1}$ and hence {\em $M_t$ is non-empty for $t\in[0, t_0]$. }

\smallskip 
 Let us next prove that $M_t$, for $t\in(0,T(M_0))$, is a smooth solution of IMCF.  First, note that $\hat M_0 \subset \text{int}(\hat M_{t})$ for $t>0$: 
indeed, since $M_0$ is locally in $C^{1,1}$, for every point $p\in M_0$,  there is an inscribed sphere at $p$ whose largest radius depends on the local $C^{1,1}$ norm
of $M_0$.  By the comparison principle between the sphere solution running from this inscribed sphere and $M_{i,t}$, we conclude that $p\in \text{int}(\hat M_t)$ for all $t>0$. {\color{black}(In practice, if $N_t=\p\hat N_t$ is a smooth compact solution containing the origin and $\hat N_0 \subset \hat  M_\tau$ some $\tau\ge0$, then the comparison and Lemma \ref{lem-scaling} imply $(1-\e)\hat N_t \subset \hat M_{i,\tau+t}$ for $i\ge i_\e$, showing $(1-\e)\hat N_t \subset \hat M_{\tau+t}$. We then take $\e \to 0$ to conclude $\hat N_t \subset \hat M_{\tau+t}$.)
}

\begin{figure}
\centering
\def\svgscale{1}{
\begingroup%
  \makeatletter%
  \providecommand\color[2][]{%
    \errmessage{(Inkscape) Color is used for the text in Inkscape, but the package 'color.sty' is not loaded}%
    \renewcommand\color[2][]{}%
  }%
  \providecommand\transparent[1]{%
    \errmessage{(Inkscape) Transparency is used (non-zero) for the text in Inkscape, but the package 'transparent.sty' is not loaded}%
    \renewcommand\transparent[1]{}%
  }%
  \providecommand\rotatebox[2]{#2}%
  \newcommand*\fsize{\dimexpr\f@size pt\relax}%
  \newcommand*\lineheight[1]{\fontsize{\fsize}{#1\fsize}\selectfont}%
  \ifx\svgwidth\undefined%
    \setlength{\unitlength}{314.19661766bp}%
    \ifx\svgscale\undefined%
      \relax%
    \else%
      \setlength{\unitlength}{\unitlength * \real{\svgscale}}%
    \fi%
  \else%
    \setlength{\unitlength}{\svgwidth}%
  \fi%
  \global\let\svgwidth\undefined%
  \global\let\svgscale\undefined%
  \makeatother%
  \begin{picture}(1,0.46267116)%
    \lineheight{1}%
    \setlength\tabcolsep{0pt}%
    \put(0,0){\includegraphics[width=\unitlength,page=1]{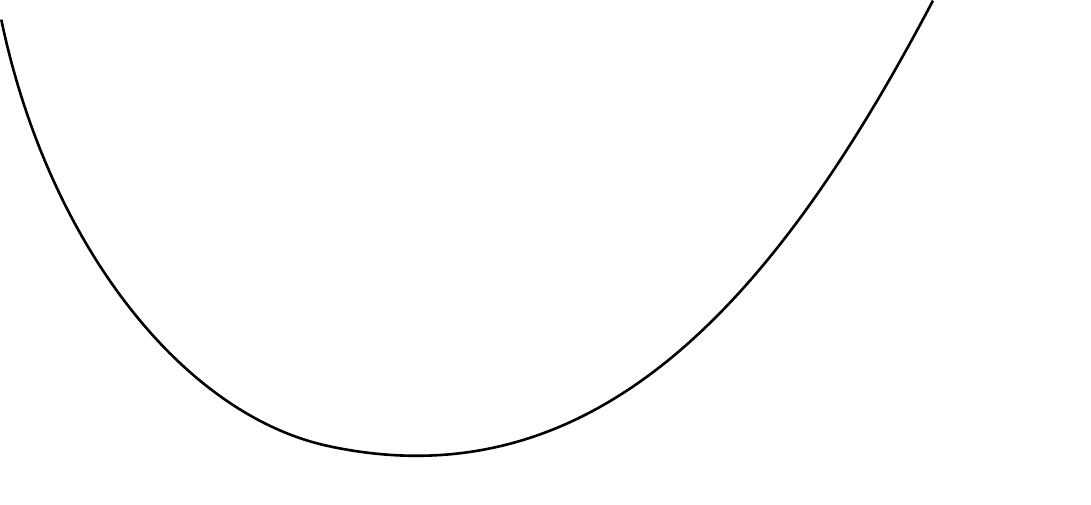}}%
    \put(0.69784255,0.31407201){\color[rgb]{0,0,0}\makebox(0,0)[lt]{\lineheight{1.25}\smash{\begin{tabular}[t]{l}$M_{t_0}$\end{tabular}}}}%
    \put(0,0){\includegraphics[width=\unitlength,page=2]{6.pdf}}%
    \put(0.32472127,0.00777922){\color[rgb]{0,0,0}\makebox(0,0)[lt]{\lineheight{1.25}\smash{\begin{tabular}[t]{l}$x_0=0$\end{tabular}}}}%
    \put(0,0){\includegraphics[width=\unitlength,page=3]{6.pdf}}%
    \put(0.6167699,0.39520365){\color[rgb]{0,0,0}\makebox(0,0)[lt]{\lineheight{1.25}\smash{\begin{tabular}[t]{l}$M_{t_0-\tau}$\end{tabular}}}}%
    \put(0,0){\includegraphics[width=\unitlength,page=4]{6.pdf}}%
    \put(0.3241499,0.18352206){\makebox(0,0)[lt]{\lineheight{1.25}\smash{\begin{tabular}[t]{l}$r_1e_{n+1}$\end{tabular}}}}%
    \put(0.3345235,0.28163904){\makebox(0,0)[lt]{\lineheight{1.25}\smash{\begin{tabular}[t]{l}$r_1-L\tau$\end{tabular}}}}%
    \put(0.70006656,0.01067049){\makebox(0,0)[lt]{\lineheight{1.25}\smash{\begin{tabular}[t]{l}$\{x_{n+1}=0\}$\end{tabular}}}}%
  \end{picture}%
\endgroup%
{
\caption{Outside hyperplane and inside sphere barriers around $x_0 \in M_{t_0}$}}}
\end{figure}

\smallskip 

Next goal is to show, for each $(x_0,t_0)\in\mathbb{R}^{n+1}\times(0,T(M_0))$ with $x_0 \in M_{t_0}$, there is a  spacetime neighborhood, say $U\times[t_0-\tau,t_0]$, around $(x_0,t_0)$ such that the portions of $M_{i,t}$ in this neighborhood can be represented as graphs over a fixed hyperplane with uniformly bounded $C^1$ norm. We may assume that $x_0=0$ and that $\{x_{n+1}=0\}$ is a supporting hyperplane of $\hat M_{t_0}$ satisfying $\hat M_{i,t} \subset \{x_{n+1}\ge 0\} $ for $t\le t_0$. For the discussion below we refer the reader to Figure 3. 
The observation in the previous paragraph says $2r_0:=\dist(\hat M_0,0)>0$. Note that the estimate on $H$ shown in Proposition \ref{lem-DH} holds  even when $M_0 \cap B_r(x_0)$ is empty. Thus, applying this proposition gives that  $H$ is bounded by $c_n r_0^{-1}$ on $M_{i,t}\cap B_{r_0}(0)$ for all $i$ and $t>0$. Recall that for smooth convex hypersurfaces, one has $|A|^2 \le H^2$. Since a (local) uniform limit of smooth functions with bounded $C^2$ norm is in $C^{1,1}$, $M_{t_0}$ has to be  $C^{1,1}$ around $0$ and hence there is some $r_1$ such that $B_{r_1}(r_1 e_{n+1}) \subset \hat M_{t_0}$. Let us choose $r_1$ sufficiently small so that $B_{r_1}(r_1 e_{n+1}) \subset  B_{r_0}(0).$ The  uniform bound  on $(H|F-v_j'|)^{-1}$ in \eqref{eq-HFv}  implies that  there is $L$ such that $H^{-1} \le L$ on $B_{r_0}(0)\cap M_{i,t}$,  for $t\in[t_0/2,t_0]$. Using this speed bound, 
we obtain that   \[ B_{r_1-L\tau}(r_1 e_{n+1}) \subset \hat M_{t_0-\tau} \quad \text{ for all }\tau \in[0, \min(\tfrac{t_0}2, \tfrac{r_1}L)].\] To prove this, let us define, for $-\tfrac{t_0}2\le s\le0$, \bee d(s):= \dist (r_1 e_{n+1} , M_{t_0+s})\text{ and }d_i(s):= \dist (r_1 e_{n+1} , M_{i,t_0+s}).\eee Note that $d_i(s)$ is a Lipschitz function and the bound on $H^{-1}$  gives \be\label{eq-1245}0\le  \dot d_i (s) \le L\quad \text { if } s\in [- \tfrac{t_0}{2},0]\quad \text{ and }\quad   r_1e_{n+1} \in \hat M_{i,t_0+s}.\ee Since $d(s)=\lim_{i\to\infty}d_i(s)$, $d(s)$ is Lipschitz and \eqref{eq-1245} holds for $d(s)$ as well.  Since $d(0) =r_1$ and $r_1e_{n+1} \in \hat M_{t_0+0}$, we may integrate \eqref{eq-1245} to conclude \[d(s)\ge r_1 + Ls \quad\text{ and }\quad r_1e_{n+1} \in \hat M_{t_0+s}\quad\text{ for all } s\in[-\min(\tfrac{t_0}2, \tfrac{r_1}L),0].\]

\smallskip

  In summary, we have shown that there are positive $r'$, $h'$, and $\tau'$ such that for  $i>i'$  large, we have  \[     B_{r'}(h'e_{n+1})  \subset   \hat M_{i,t} \subset \{x_{n+1}\ge 0\}\text{ for all }t\in[t_0-\tau',t_0].\] Therefore: {\em  if $D_{r'} := \{ x' \in \mathbb{R}^{n} \,:\, |x'|\le r'\}$, we can write $M_{i,t} \cap (D_{r'} \times [-h',h'])$ as a graph $x_{n+1} = u^{(i)} (x',t)$ on $D_{r'} \times [t_0-\tau', t_0]$. } Note we have a uniform bound of $|D_x u|$ on $D_{r'/2}$ by the ball and hyperplane barriers above and below. Since $M_{i,t}$ are solutions to  IMCF, the  graphs $x_{n+1}=u^{(i)}(x',t)$ evolve by the  fully nonlinear parabolic equation \be\label{eq-IMCFgraph}\p_t u =- \fr{(1+|Du|^2)^{1/2} }{H}= -(1+|Du|^2)^{1/2} \left[\text{div} \left(\fr{Du}{(1+|Du|^2)^{1/2}} \right)\right]^{-1}\ee and the equation is uniformly parabolic if $|Du|$, $H$, $H^{-1}$ are bounded. Therefore, our estimates above show that $u^{(i)}$ are solutions to a uniformly parabolic equation on $D_{r'/2}\times [t_0-\tau',t_0]$ and moreover  they are uniformly bounded, since $|u^{(i)}|\le h'$. Standard parabolic regularity theory implies a smooth subsequential convergence $u^{i} \to u $ on  $D_{r'/4}\times [t_0-\tau'/2,t_0]$.  Since the sequence of surfaces $M_{i,t}$ is  monotone in $i$, this proves that  $x_{n+1} = u(x',t)$ is a smooth graphical parametrization of $M_t$ which is a solution to  \eqref{eq-IMCFgraph}.  Our argument  holds in a  neighborhood around any point  $x_0\in M_{t_0}$  and for any  $t_0\in(0,T(M_0))$, therefore showing that   $M_t $ is a smooth solution to the IMCF for $t\in (0,T)$.  

\smallskip 

We will next obtain the  local  uniform convergence of $M_t \to M_0$,  as $t \to 0$,   by showing that $\cap_{t>0}\hat M_t  = \hat M_0$. Arguing by contradiction, suppose that $\hat M_0 \subsetneq \cap_{t>0}\hat M_t$, that is  there exists a point $p\in \text{int}( \cap_{t>0}\hat M_t)\setminus \hat M_0$. This means for each $t>0$, there is $i_t$ such that  $\hat M_{i,t}$ contains $p$ if $i>i_t$.  Let us define $d_i(t):= \dist(p, \hat M_{i,t})$.  Note that $d_i(0)>0$ and $d_i(t)$ is nonnegative  decreasing  function.  In view of the bound  \eqref{eq-HFv}, by choosing $t_0=T/2$, we conclude that  there is $C=C(p,M_0)>0$ such that if $0<t<T/2$, then 
\be\label{eqn-dit} \dot d_i(t)\ge - C\, t^{-1/2}.
\ee Furtermore,  the function  $d_i(t)$ is  Lipschitz continuous   and hence the above inequality holds a.e. One  obtains  this inequality by considering those points attaining the distance at each fixed time and estimate the rate of changes in the distances between those points and $p$ using the bound on $H^{-1}$. Integrating \eqref{eqn-dit}  from $0$ to $t$, we get   $d_i(t)- d_i(0) \ge -2C \, t^{1/2}$,  for $0<t<T/2$. Note that $d_i(0) \ge \dist(p,\hat M_0) >0$. There is $t_1>0$ such that $d_i(t) > \dist(p,\hat M_0)/2$ for all $i$ and $t\in(0,t_1)$. This is a contradiction to the assumption which says $d_i(t) = 0$ for large $i>i_t$. 
 
\smallskip 
Finally, we prove the strict convexity assertion in Theorem \ref{thm-mainexistence} using Appendix \ref{sec-strict}. If $\hat M_0$ contains an infinite line, a solution at later time $\hat M_t$ also contains the same line and hence $M_t$ it is not strictly convex by Lemma \ref{lem-elementary}.  Now suppose $\hat M_0$ has no infinite straight.  {\color{black}In view of Corollary \ref{cor-strict}, it suffices to show  $\mathcal{H}^n (\nu [M_{t}] )>0$ for all $t\in(0,T)$. Let us fix an arbitrary $t_0 \in (0,T)$.}
In the construction of $\hat M_t$, we have shown that $M_{i,t_0}$ (and hence $M_{t_0}$) are contained in some round cone $\hat C :=\{x\in\mathbb{R}^{n+1}\,:\, \la x-v, \omega \ra \ge (\sin\delta) |x-v| \}$.  {\color{black}Observe that the outward normal of each supporting hyperplane of $\hat C$ should belong to $\nu [M_{t_0}]$ as we may translate the hyperplane to make it support $\hat M_{t_0}$ somewhere. We directly compute the set of outward normals of supporting hyperplanes of $\hat C$ as  $\{ v \in \mathbb{S}^n\,:\,  \la v, -\omega \ra \ge \cos \delta \}=:\hat \Gamma'$. This shows $\mathcal{H}^n (\nu [M_{t_0}] )\ge \mathcal{H}^n(\hat \Gamma')>0$}, finishing the proof.

 \end{proof}

{\color{black}\begin{remark} \label{remark-candidate}
The definition of {\em innermost candidate} in \eqref{eq-candidate} does not depend on the choice of approximation $\hat M_{i,0}$: if $\hat M_{i,0}$ and $ \hat M'_{i,0}$ are two approximations of $\hat M_0$, we have $\hat M_{i,0} \subset\subset \hat M'_{n_i,0}$ for large $n_i$, showing that $\hat M_{i,t} \subset \cup_j \hat M'_{j,t}$ and vice versa. By the same argument, the comparison principle holds between two innermost candidates if one contains the other at initial time. The solution is innermost in the sense described in Lemma \ref{lem-noncptavoid}. This fact will be used in the remaining of the section. 
\end{remark}}

\begin{lemma}\label{lem-noncptavoid}
Let $N_t=\p \hat N_t$ for $t>0$ be a smooth solution to the IMCF with initial data $\hat N_0:=\cap_{t>0} \hat N_t$ and $\hat M_t$ {\color{black}be the innermost candidate from $\hat M_0$ by Definition \ref{def-candidate}.} If $\hat M_0 \subset \hat N_0$, then $\hat M_{t} \subset \hat N_t$ as long as $N_t=\p \hat N_t$ exists.

 \begin{proof}
$\hat M_{i,0} \subset \subset \hat N_0$ implies  $\hat M_{i,t}\subset \hat N_{t}$ by the comparison principle, showing $\hat M_{t} \subset \hat N_t$.   
\end{proof}\end{lemma}
 {\color{black}Next lemma shows conical solutions can be used as barriers from inside.}
\begin{lemma}\label{lem-conicalcomp} Let $\Gamma_0=\p \hat \Gamma_0\subset \mathbb{S}^n$ be a smooth strictly convex hypersurface in  $\mathbb{S}^n$ and $\Gamma_t$ be the unique solution to the IMCF obtained by  \cite{makowski2013rigidity} and \cite{Ge}. {\color{black}Let $\hat N_t$ the innermost candidate from $\hat N_0$ by Definition \ref{def-candidate}. If $\mathcal{C}\hat\Gamma_0 \subset \hat N_0$, then $\mathcal{C}\hat \Gamma_t \subset \hat N_t$ for $t\in[0,\ln |\mathbb{S}^{n-1}| -\ln |\Gamma_0|)$.  }

\begin{proof}
After  a rotation, we may assume that $e_{n+1}$ is in the interior of $\hat \Gamma_0$ in $\mathbb{S}^n$. Then $\mathcal{C}\Gamma_0$ can be  written as a graph of an entire homogeneous function $x_{n+1} =f (x)$, $x\in\mathbb{R}^n$, which is uniformly Lips{\color{black} ch}itz. Since the graph is a cone, we have $f(\lambda x) = \lambda f(x)$. Let $\eta $ be a usual smooth radially symmetric mollifier supported on $B_1(0)$, and define $$f_\e(x):= f*\eta_\e(x)=\int_{\R^n}  f(x+w)\eta(\e^{-1}w) \e^{-n}\, dw .$$ The convexity of this mollified function $f_\e$ can be easily checked:

\be\ba \label{eq-1fe}f_\e(\lambda x+ (1-\lambda) y) &= \int f(\lambda x+ (1-\lambda) y+w)\eta(\e^{-1}w) \e^{-n}dw  \\& \ge  \int[ \lambda f( x+w )+ (1-\lambda) f( y+w)]\eta(\e^{-1}w) \e^{-n} \, dw  \\& =\lambda f_\e ( x) + (1-\lambda)f_\e(y).\ea\ee
Moreover, $f_\e \ge f$ since \be\ba\label{eq-2fe} f_\e(x)&= \int f(x+w)\, \fr{\eta(\e^{-1}w)+\eta(-\e^{-1}w)}{2} \e^{-n}   dw\\&= \int \frac{f(x+w)+f(x-w)}{2}\eta(\e^{-1}w) \e^{-n} dw  \ge \int f(x)\eta(\e^{-1}w) \e^{-n}dw =f(x),\ea\ee
the uniform Lipschitz condition of $f$ implies that $\Vert f_\e- f\Vert_\infty <\infty $ for all $\e>0$,  and that $\Vert f_\e-f \Vert_\infty \to 0$,  as $\e\to 0$. Next,  observe that
\be\ba \label{eq-fscaling}\fr{f_1(\lambda x)}{\lambda} =  \int \fr{f(\lambda x+w)}{\lambda}\eta(w) dw =  \int f(x+\fr w\lambda)\eta(w) dw=\int f(x+y)\eta(\lambda y)\lambda^{n} dy=f_{\lambda^{-1}}(x).\ea\ee Let $M_0=\p\hat M_0$ be the convex hypersurface $\{(x,x_{n+1} ) \,:\, x_{n+1}=f_1(x) \}$. Then \eqref{eq-fscaling} implies that the tangent cone of $M_0$ at infinity is $\mathcal{C}\Gamma_0$. Theorem  \ref{thm-mainexistence} shows the existence of a  smooth solution $M_t$,  for $t\in[0, \ln |\mathbb{S}^{n-1}|-\ln |\Gamma_0|)$ with initial data $M_0$.

\smallskip 

We will show next that: {\em for $t\in[0, \ln |\mathbb{S}^{n-1}|-\ln |\Gamma_0|)$,   $\e M_t$ converges to $\mathcal{C} \Gamma_t$ in $L_{loc}^\infty$.}  Assuming this, let us first finish the proof of the lemma: for each $\e>0$, \eqref{eq-fscaling} implies $\e M_0 = \{x_{n+1} =f_\e(x)\}$ and thus \eqref{eq-2fe} implies  $\e \hat M_0 $ is contained in $\hat N_0$.  {\color{black} $\e \hat M_t$ is an innermost candidate as $\hat M_t$ is.  Thus $\e \hat M_t \subset \hat N_t$ by the comparison in Remark \ref{remark-candidate}} and, by taking $\e\to0$, we conclude $\mathcal{C}\hat \Gamma_t \subset \hat N_t$. 

\smallskip 

We are left to prove the convergence of $\e M_t$ to $\mathcal{C} \Gamma_t$.   Following the construction in Theorem \ref{thm-mainexistence}, let $M_{i,t}$ be compact convex solutions which approximate $M_{t}$ from inside. Since $\hat M_{i,0}$ is contained in $\mathcal{C}\hat \Gamma_0$, the comparison principle implies $\hat M_{i,t} \subset \mathcal{C}\hat \Gamma_t$, showing $\hat M_t \subset \mathcal{C}\hat \Gamma_t$. Let us express $M_t$ by the entire graphs $x_{n+1}=f_1(x,t)$ and $\mathcal{C}\Gamma_t$ by $x_{n+1}= f(x,t)$. Observe that the gradients $|Df_1|$ and $|D f|$ are uniformly bounded on $(x,t)\in\mathbb{R}^n\times [0, |\mathbb{S}^{n-1}|-\ln |\Gamma_0|)$. This is because $e_{n+1}$ is an interior point of $\hat \Gamma_0$ and hence $\hat M_t$ and $\mathcal{C}\hat \Gamma_t$ contain a fixed round convex cone whose axis in positive $e_{n+1}$ direction. Next, note  $f(\cdot,0)\in C^{\infty}_{loc}(\mathbb{R}^{n} \setminus \{0\})$ and $f_\e (\cdot,0)\to f(\cdot,0)$ in $C^{\infty}_{loc} (\mathbb{R}^{n} \setminus\{0\})$ as $\e\to0$. This convergence and  \eqref{eq-fscaling} imply that there is $C>0$ such that  $H ( |x|+1 ) \le C$ for $M_0$.   
Proposition \ref{lem-DH} then implies that the mean curvature $H(x,t)$ of $M_t$ at $(x,f_1(x,t))$ satisfies the bound  $H(x,t)(|x| +1)\le C'$. Next, since $\mathcal{C}\Gamma_t$ works as a conical barrier outside, Theorem \ref{thm-mainest3} (see also Remark \ref{rem-initial}) can be applied to the approximating compact solutions $M_{i,t}$ to conclude that $(H|F|)^{-1} \le C_\delta$ on $M_t$ for $t\in[0,T(M_0)-\delta]$.
  
All the  bounds above imply that  the solutions $\e M_t$ when viewed as entire graphs, have uniform gradient bounds, locally uniform height bounds, and locally uniform bounds of $H$ and $H^{-1}$ on compact domains which do not contain the origin. In the previous statement, the uniformity of estimates holds both in $\e>0$ and $t\in[0,T(M_0)-\delta]$ for all fixed $\delta>0$.  By the regularity estimates of uniformly parabolic equations, we may pass to a sub-sequential limit {\color{black}and obtain, as $\e \to 0$, $\e f(\e^{-1} x,t )$ converges to some  $f_0 (x,t)$ smoothly on $(B_{\delta^{-1}}\setminus B_\delta )\times [0,T-\delta)$ for all $\delta>0$.} It follows that $\{x_{n+1} = f_0(x,t)\}$ is a smooth solution to IMCF on $\R^n \setminus \{ 0 \}$.  Meanwhile,  $\{x_{n+1} = f_0(x,t)\}$ represents a convex cone as it is the blow-down of $M_t$. Combining these together, $\{x_{n+1} = f_0(x,t)\}= \mathcal{C} \Gamma'_t$ for some smooth convex hypersurface $\Gamma'_t \subset \mathbb{S}^n$ and $\Gamma'_t$ evolves by the IMCF.  Since $\Gamma_t $ is the unique solution to IMCF with initial data $\Gamma_0 = \Gamma_0'$, we conclude that  $\Gamma_t =\Gamma'_t$.  This proves that $f_0(x,t) = f(x,t)$ and the convergence of $\e M_t$ to $\mathcal{C}\Gamma_t$.

\end{proof}
\end{lemma}

 \begin{prop}\label{prop-interior} {\color{black} Let $M_t=\p \hat M_t$ be a convex non-compact solution obtained from Theorem \ref{thm-mainexistence} and $\hat \Gamma_t$ be the link of the tangent cone of $\hat M_t$. Suppose $\hat \Gamma_0$ has no interior in $\mathbb{S}^{n}$ but $\mathcal{H}^{n-1}(\hat \Gamma_0)>0$. Then $\hat\Gamma_t$ has interior for $t>0$.}

\begin{proof}\begin{figure}\centering
\def\svgscale{1}{
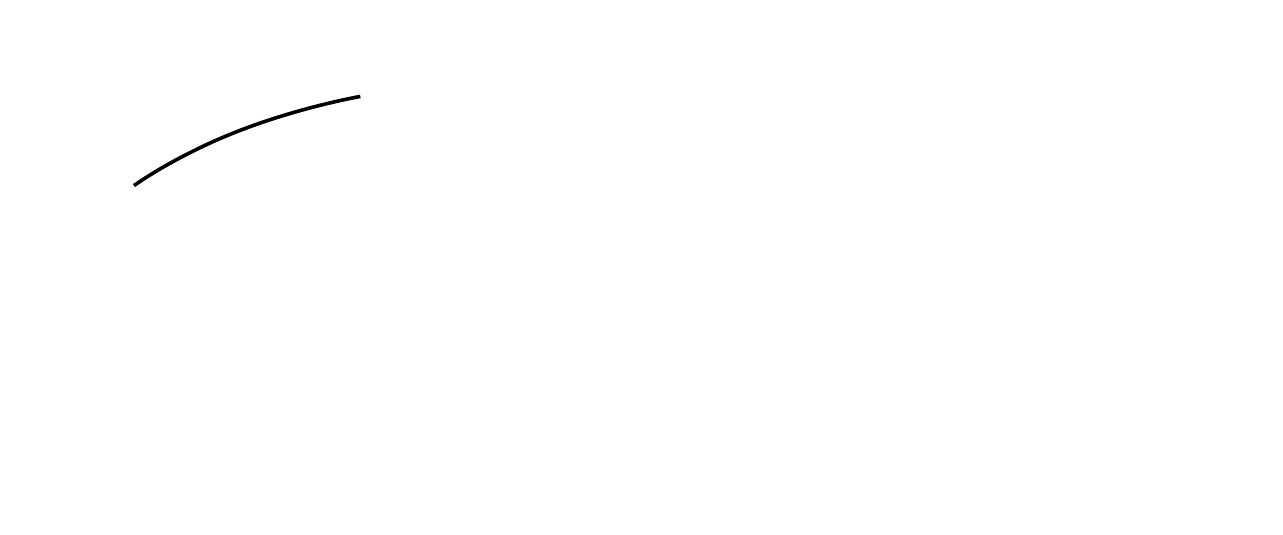
\caption{Proposition \ref{prop-interior}}}
\end{figure} {\color{black} We may assume that $\mathcal{H}^{n-1} (\hat \Gamma_0 \cap \{x_{n+1}>0\}) >0$ and 
$0\in \text{int}(\hat M_0)$, by rotating and translating the coordinates respectively. Since $\mathcal{C}\hat\Gamma_0$ is  convex in $\mathbb{R}^{n+1}$,  $\mathcal{C}\hat \Gamma_0 \cap\{x_{n+1}= 1\}=: \hat \Omega$ is convex in $\mathbb{R}^n$. Moreover, $\hat \Omega $ has no interior in $\mathbb{R}^{n}$ and $\mathcal{H}^{n-1} (\hat \Omega ) >0$ since the gnomonic projection  \eqref{eq-varphi} is (locally) bi-Lipschitz map between $\mathbb{R}^n$ and $\mathbb{S}^n \cap \{x_{n+1}>0\}$. If a convex set in Euclidean space has no interior, then it should be contained in some hyperplane. Therefore, $\hat \Omega$ is contained in a $(n-1)$-plane and $\hat \Omega$ has interior in that $(n-1)$-plane as otherwise $\hat \Omega$ would be contained in a $(n-2)$-plane, contradicting $\mathcal{H}^{n-1} (\hat \Omega ) >0$. As a result, $\hat \Omega$ in $\mathbb{R}^n\times\{1\}$ contains a $(n-1)$-dimensional disk of some radius $r_0$.} The cone generated by this $(n-1)$-disk is contained in $\mathcal{C} \hat \Gamma_0$ and thus $\mathcal {C}\hat \Gamma_0$ should contain a rotated image of the following $n$-dimensional cone for some $0<r\le r_0$: \[ \mathcal{C}:=\{(x_1,x_2,\ldots,x_{n},0)\in\mathbb{R}^{n+1}\,:\, \sqrt{x_1^2 +\cdots +x_{n-1}^2} \le r x_{n} \}.\]  
By letting  $c :=r/\sqrt{1+r^2}$, observe $B^n_{ca}(a e_{n+1}) \times \{0\}$ is contained in $\mathcal{C}$ for all $a>0$.   Since we assumed $0\in \text{int}(\hat M_0)$, there are $\vec{v} \in \hat \Gamma_0\subset \mathbb{S}^n$, $\e>0$, and a rotation operator ${J}$ such that the family of expanding thin disks $a\vec{v}+ {J}(B^n_{ca}(0)\times (-\e,\e))$ are contained in $\hat M_0$ for all $a>0$.  
\begin{claim}\label{claim-barrier} Let $\hat D_{R,\e} := B^{n}_R(0) \times (-\e,\e) \subset \R^{n+1}$ be a thin disk $\e\in (0,R/100)$ and $N_t=\p \hat N_t$ be a smooth IMCF. If $\hat D_{R,\e}\subset\hat N_0$, then there is $c_n>0$ such that $B^{n+1}_{c_n Rt}(0)\subset \hat N_t$ for $t\in[0,c_n]$. 
\begin{proof}[Proof of Claim]
Let us smoothen out the edges of $D_{R,\e}:= \partial \hat D_{R,\e}$ to obtain a smooth pancake like convex hypersurface $\Sigma_{R,\e}$ by a similar method of Lemma \ref{lem-conicalcomp}: consider the convex conical hypersurface in $\mathbb{R}^{n+2}$ generated by $D_{R,\e}\times \{1\} \subset \mathbb{R}^{n+2}$ from the origin. We can represent this cone by an entire graph $x_{n+2}=f(x)$ of a $1$-homogeneous function $f(x)$. Then 
$f^{-1}(\{1\}) = D_{R,\e}$. If we consider the regularization of $f$, say $f_\delta$, as constructed in Lemma \ref{lem-conicalcomp}, then $f_\delta \ge f$ and it is smooth convex function. For sufficiently small $\delta$, the level set $f_{{\color{black}\delta}}^{-1}(\{1\})=:\Sigma_{R,\e}$ is a smooth convex hypersurface in $\mathbb{R}^{n+1}$ which is contained in $\hat D_{R,\e}$. Since the regularization of a linear function is the same as itself,  $D_{R,\e}$ and $\Sigma_{R,\e}$ coincide on $B^{n+1}_{R/2}(0)$ for small $\delta>0$.
  
Observe that  $\Sigma_{R,\e}$  has the same symmetry as $D_{R,\e}$.  i.e.  $O(n)$-rotational symmetry and reflection symmetry with respect to $\{ x_{n+1} =0\}$. Thus the IMCF  $\Sigma_{R,\e}(t)$ starting at $\Sigma_{R,\e}$ must contain  two points $(0,\e+c(t))$ and $(0,-\e-c(t))$, for each $t>0$, at which the  normal vectors to $\Sigma_{R,\e}(t)$ are $e_{n+1}$ and $-e_{n+1}$, respectively. In view of Lemma \ref{lem-DH}, $c'(t) > c R$ as long as $\e+c(t) <R/2$. Since $\Sigma_{R,\e}(t)$ contains these two points and the  disk $B^n_{R/2} \times \{0\}$, convexity implies that $\hat \Sigma_{R,\e}(t)$ includes  our desired ball. This finishes the proof as $\hat \Sigma_{R,\e}(t)\subset \hat N_t$ by the comparison principle.
 \end{proof} 
 \begin{figure}
\centering
\def\svgscale{1}{
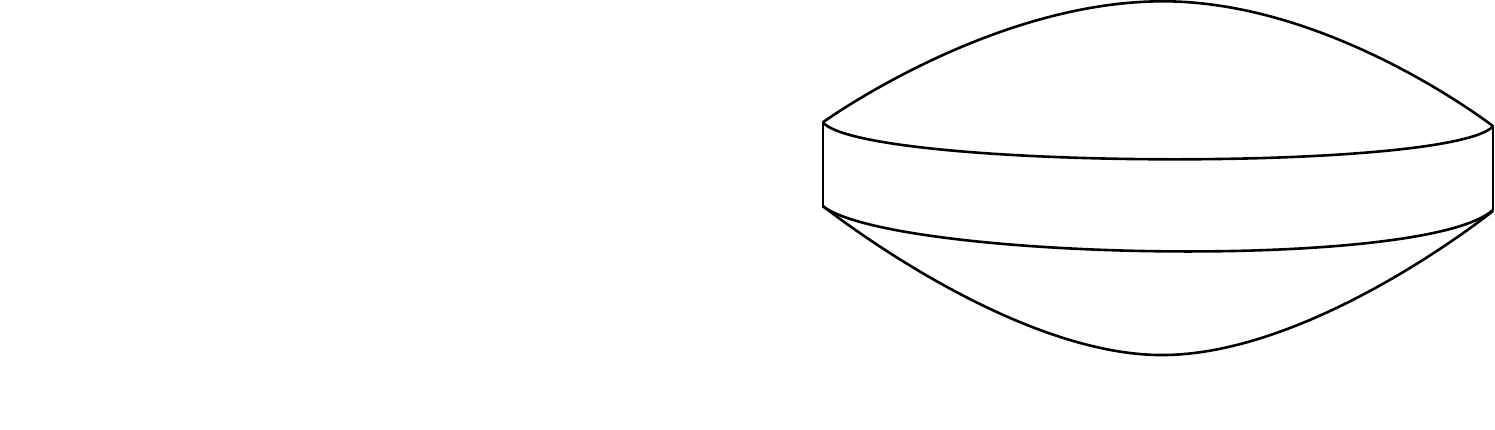
\caption{Claim \ref{claim-barrier}}\label{figure1}}
\end{figure}
\end{claim} 
 By the claim, $a\vec{v}+B^{n+1}_{c_n c  at}(0) \subset \hat M_t$ for all $0\le t \le c_n$ and $a \ge a_0$ some $a_0>0$. $\hat \Gamma_t$ has  interior since $\hat M_t$ contains a round cone.

\end{proof}

\end{prop}

{\color{black} We now prove Theorem \ref{thm-noextension}. Note that the same proof works for both $T=0$ and $T\in(0,\infty)$. 
    \begin{proof}[Proof of Theorem \ref{thm-noextension}] 
    Let $\hat M_t$, for $t\ge0$, be the innermost candidate solution from $\hat M_0$ (see  Definition \ref{def-candidate}). Since $\hat M_t \subset \hat N_t$ (by Lemma \ref{lem-noncptavoid}), it suffices to show $\hat M_{T+\tau} = \mathbb{R}^{n+1}$ for $\tau>0$.    
%
%
%
%
%
%
 One useful observation is that if $\hat M_{t_0}$ contains a half space at $ t_0\ge 0$ then $\hat M_t=\mathbb{R}^{n+1}$ for $t>t_0$: suppose $\{x_{n+1}\ge 0\}\subset \hat M_{t_0}$. By the comparison with spherical solutions, $\p B_{re^{\tau'/n}} ( r e_{n+1})\subset \hat M_{t_0+\tau'}$. Since $ B_{r (e^{\tau'/n}-1)}(0) \subset B_{re^{\tau'/n}} ( r e_{n+1})\subset \hat  M_{t_0+\tau'}$ for all $r>0$, we get $\hat M_{t_0+\tau'}= \mathbb{R}^{n+1}$.    
    
We may assume $0\in \text{int}(\hat M_0)$. Suppose $\hat M_{t_0}$ contains a cone $\mathcal{C} \hat \Gamma'_0$ with a smooth strictly convex link $\Gamma'_0$ in an open hemisphere. Since the IMCF running from $\Gamma'_0$ converges to an equator as $t\to \ln |\mathbb{S}^{n-1}| -\ln |\Gamma'_0|$ (by \cite{Ge}\cite{makowski2013rigidity}), Lemma \ref{lem-conicalcomp} and the observation above imply that $\hat M_t =\mathbb{R}^{n+1}$ for $t> t_0 + \ln |\mathbb{S}^{n-1}| -\ln |\Gamma'_0|$. In view of the approximation in Lemma \ref{lem-outersphere}, the same assertion holds when $\hat \Gamma'_0$ is a convex set with interior and contained in an open hemisphere.

%
%
%

 In general, the link of initial tangent cone $\hat \Gamma_0$ is a convex set in a closed hemisphere. After a rotation, we may assume that $e_{n+1}\in \hat \Gamma_0$ and represent $\hat M_0$ using a convex function $f$ on a convex domain $\Omega_0 \subset \mathbb{R}^n$ by $\hat M_0= \{(x,x_{n+1}) \,:\,  x_{n+1}\ge  f (x), x\in \Omega_0 \}$.  Let us define $\hat M_{\e,0}:=  \{(x,x_{n+1}) \,:\,  x_{n+1}\ge  f (x)+ \e \sqrt{|x|^2+1}, x\in \Omega_0 \}$ and observe that  it is still convex with $C^{1,1}_{loc}$ boundary.  The links of the tangent cones of $\hat M_{\e,0}$, say $\hat \Gamma_{\e,0}$, are contained in a fixed open hemisphere and $\hat \Gamma_{\e,0}$ increases to $\hat \Gamma_0$ as $\e\to0^+$. $P(\hat \Gamma_{\e,0})$ is monotone in $\e$ by Lemma \ref{lem-genouter}. Let us assume the following claim for the  moment. 
   \begin{claim} $P(\hat \Gamma_{\e,0})$ increases to $P(\hat \Gamma_0)$ as $\e\to 0^+$. \end{claim}Choose $\hat M_{\e',0}$ such that  $T':=\ln |\mathbb{S}^{n-1}| - P(\hat \Gamma_{\e',0})< T+\tau/2$ and note $T'>0$ by Lemma \ref{lem-genouter}. By Theorem \ref{thm-mainexistence}, there is a smooth strictly convex solution $M_{\e',t}$ for $0<t< T'$.  At $t_0:=\min(\tau/4,T'/2)$, the link of the tangent cone of $\hat M_{\e',t_0}$, say $\hat \Gamma_{\e',t_0}$, has interior by Proposition \ref{prop-interior}. $\hat \Gamma_{\e',t_0}$ is contained in an open hemisphere due to strict convexity of  $M_{\e',t_0}$.  By Lemma \ref{lem-genouter},  $T'':= \ln |\mathbb{S}^{n-1}| - \ln |\Gamma_{\e',t_0}|\le \ln |\mathbb{S}^{n-1}| - P(\hat \Gamma_{e',0})<T+\tau/2$. $\hat M_{\e,t_0} \subset \hat M_{t_0}$ by Lemma \ref{lem-noncptavoid}. Because $\mathcal{C} \hat \Gamma_{\e',t_0} \subset \hat M_{t_0}$, we apply the assertion in the second paragraph to conclude that $\hat M_t =\mathbb{R}^{n+1}$ for $t>t_0 + T''$.  Note that $t_0+T'' < T+3\tau/4$  finishing  the proof.
     
   \begin{proof}[Proof of Claim] Let us define a locally Lipschitz map \bee\psi_\e: \mathbb{R}^{n+1}- \{0\}\longrightarrow \mathbb{R}^{n+1} \quad \text{ by }\quad\psi_\e(x,x_{n+1}) = \frac{(x,x_{n+1} +\e |x|)}{|(x,x_{n+1}+\e |x|)|},  \eee
  and observe $\hat \Gamma_{\e,0} = \psi _\e (\hat \Gamma_0)$. $\psi_\e$ induces a bijection between $\{ x_{n+1} = \alpha |x|\}\cap \mathbb{S}^n$ and $\{ x_{n+1} = (\alpha+\e) |x|\}\cap \mathbb{S}^n$ for all $\alpha \in \mathbb{R}$, hence   $\psi_\e$ induces a bijection from and onto $\mathbb{S}^n$. Let us define $ \phi (x,x_{n+1}) := |x|e_{n+1}$ so that $ \psi_\e (p) = \frac{ p +\e \phi(p)}{|p+ \e\phi(p)|}$. At each differentiable point of $\psi_\e$, the differential of $\psi_\e$, $D \psi_\e$, represented by a $(n+1)\times(n+1)$ matrix is     
   
   \[D \psi_\e (p) = \frac{1}{|p+\e \phi |} ( I_{n+1} + \e D \phi) - \frac{p+\e \phi }{|p+\e\phi|^3 } (p+\e\phi)^{T} (I_{n+1}+\e D\phi). \]
  Using $|p|=1$, $|\phi|\le 1$ and $|D\phi|\le 1$ on $p\in \mathbb{S}^n$,  we may find positive $C$ and $\e_0$ such that for $0<\e \le \e_0$ at each  $p\in \mathbb{S}^n-\{\pm e_{n+1}\}$ and $V \in T_p \mathbb{S}^n = \{W \in \mathbb{R}^{n+1} \,:\, \la W, p \ra =0 \}$ $$(1-C\e)\vert V\vert\le  \vert (d\psi_\e)_p (V)\vert = \vert D\psi_\e(p) V \vert  \le (1+C\e) \vert V\vert  .$$ 
The shows that $(1-C\e)^{n-1} |\Gamma_{0}| \le |\Gamma_{\e,0}| \le (1+C\e)^{n-1} |\Gamma_{0}|$.

   \end{proof}

     \end{proof}
    }

In the remaining part of this section, we assume $0<T(M_0)<\infty$ and discuss the behavior of innermost solution $\hat M_t$ as $t\to T^-$. By Theorem \ref{thm-mainexistence}, $M_t$ is a smooth IMCF for $0<t<T$ which is innermost by Lemma \ref{lem-noncptavoid}. Observe $\hat \Gamma_t$, the link of the tangent cone of $ \hat M_t$, is a kind of weak solution to the IMCF on $\mathbb{S}^n$. 

\begin{lemma} \label{lem-compgamma}Suppose $\hat \Gamma'_{0}\subset  \hat \Gamma_{0}$ for some compact convex set $\hat \Gamma'_{0}$ with smooth strictly convex boundary and let $\Gamma_t'$ be the IMCF from $\Gamma_{0}'$. Then,  $\hat \Gamma_{t}' \subset \hat \Gamma_{t}$. Similarly, if $\hat \Gamma_{0}\subset  \hat \Gamma_0'$ for some $\hat \Gamma_0' \neq \mathbb{S}^n$ with smooth strictly convex boundary, then $\hat \Gamma_t \subset \hat \Gamma'_{t}$. The inequalities hold as long as $\Gamma_t'$ exists.

\begin{proof} For such a $\hat \Gamma'_0 \subset \hat \Gamma_0$, we may find a vector $v\in\mathbb{R}^{n+1}$ such  that $ \mathcal{C} \hat \Gamma'_0+v\subset \hat M_0$. Then Lemma \ref{lem-conicalcomp} implies that $ \mathcal{C} \hat \Gamma'_t+v\subset \hat M_t$,
hence   $\hat \Gamma'_t \subset \hat \Gamma_t$. 

\smallskip

In  the other case, if we first assume  strict inclusion  $\hat \Gamma_0 \subset \subset \hat \Gamma'_0$, then  the proof goes similarly except that we don't use Lemma \ref{lem-conicalcomp} and use the  usual comparison principle between the approximating compact solutions $M_{i,t}$ and the conical barrier outside (and then take $i\to\infty$).   For general $\hat \Gamma_0 \subset \hat \Gamma'_0$, consider a strictly  decreasing approximating sequence  $\{ \hat \Gamma'_{i,0} \}$ of $\hat \Gamma'_0$, obtained from Lemma \ref{lem-outersphere}, and apply the result which assumes strict  inclusion. This shows that $\hat \Gamma_t \subset \cap_{i} \hat \Gamma'_{i,t}$ and $\Gamma'_t \subset  \cap_{i} \hat \Gamma_{i,t}'$. In view of Lemma \ref{lem-outersphere}, we have the following equalities
$$|\p(\cap_{i} \hat\Gamma_{i,t}')| = \lim_{i \to \infty} |\Gamma_{i,t}'| =\lim_{i\to\infty} e^t |\Gamma_{i,0}'|= e^t |\Gamma_0'| = |\Gamma_t'|$$ and the strict outer area minimizing property (the last assertion in Lemma \ref{lem-outersphere}) implies $\hat \Gamma_t' =  \cap_{i} \hat \Gamma_{i,t}'$, and this shows $\hat\Gamma_t \subset \hat\Gamma_t'$.

\end{proof}

\end{lemma}

\begin{remark}[Asymptotic behavior of $M_t$, as $t \to T^-$] \label{remark-blowdown}In the case when $\Gamma_0$ admits a smooth strictly convex IMCF, the comparison {\color{black}in Lemma \ref{lem-compgamma}} implies that $\Gamma_t$ 
must  be this solution and, by the result of \cite{makowski2013rigidity} or \cite{Ge}, $\Gamma_t$ converges to an equator in $C^{1,\alpha}$,  as $t \to  T^-= \ln| \mathbb{S}^{n-1}|- \ln P(\hat \Gamma_0)$. In other words,  the solution $M_t$ becomes flat, as $t\to T^{-}$, in the sense that the tangent cone $\mathcal{C} \Gamma_t$ converges to a hyperplane in $C^{1,\alpha}_{loc}$. Next, since $\cup_{t<T} \mathcal{C}\hat\Gamma_t\subset \mathcal{C}\hat \Gamma_T \subset \hat M_T+v $,  for some $v\in\mathbb{R}^{n+1}$, $\hat M_T+v$ must  contain the closure of $\cup_{t<T} \mathcal{C}\hat\Gamma_t$, a closed half space. Thus $\hat M_T$ is either a closed half space or  $\mathbb{R}^{n+1}$. In the first case,  $M_t$ converges to a hyperplane as $t\to T^-$,  and in the second case, $M_t$ disappears to infinity as $t\to T^-$. Furthermore,  in the first case  the convergence to the hyperplane is in $C^{1,\alpha}_{loc}$ as we have a locally uniform $C^2$ bound (see Proposition \ref{lem-DH}). {\color{black}Note this result resembles $C^{1,\alpha}$ convergence of flow with boundary to a flat disk shown in \cite{LS}.} One additional condition on $\hat M_0$ which {\color{black}leads to the first case} is to {\color{black}assume}   there is  $v\in \mathbb{R}^{n+1}$ such that $\hat M_0 \subset \mathcal{C}\hat \Gamma_0 + v$. Then $\mathcal{C}\hat \Gamma_t +v$ becomes a barrier which contains $\hat M_t$ and thus $\hat M_T$ has to be a half space. 

In general $\Gamma_0$ may not admit  a smooth IMCF solution. This gave the  motivation for further research by the first author and Pei-Ken Hung \cite{CH} and one of the results in \cite{CH} shows the following: there is $t_0< T:=\ln |\mathbb{S}^{n-1}| - \ln P(\hat \Gamma_0)$ and $0\le k \le n-2$ such that $\mathcal{C}\Gamma_t = \mathbb{R}^k \times \mathcal{C}\Gamma'_t$ and   $\Gamma'_t$ is a IMCF in $\mathbb{S}^{n-k}$ which becomes smooth {strictly convex} for $t\in(t_0,T)$.  Therefore, in this general case one obtains  the same asymptotic behavior as in the case of previous paragraph. 

If we do not assume the result of \cite{CH}, we {\color{black}still have some partial results} on the asymptotic behavior of the solution.  $\{ \Gamma_t \}_{\in(0,T)} $ is a monotone family of convex hypersurfaces which are all contained in  some hemisphere and that  $|\Gamma_t|=e^t P(\hat \Gamma_0)$. Therefore, the closure of $\cup_{t<T} \hat \Gamma_t$ is a convex set in a (closed) hemisphere whose outer area is the same as the area of an equator $|\mathbb{S}^{n-1}|$. Such a convex set is either a  hemisphere or a  wedge discussed in \eqref{eq-wedge}. Since $\cup_{t<T} \hat \Gamma_t$ is contained in $\hat \Gamma_T$, $\mathcal{C}\hat \Gamma_T$ is either a  wedge, a half space, or $\mathbb{R}^{n+1}$. In the first case (although there is no such case if the result of \cite{CH} is assumed), $M_T= \Sigma \times \mathbb{R}^{n-1}$ for some non-compact $C^{1,1}_{loc}$ convex curve $\Sigma=\p \hat \Sigma$ in $\mathbb{R}^2$ and $M_t$ converges, as $t \to T^-$,  to $\Sigma \times \mathbb{R}^{n-1}$  in $C^{1,\alpha}_{loc}$. The cases when $\mathcal{C} \hat \Gamma_T$ is a half space or $\mathbb{R}^{n+1}$ were described   the above.

\end{remark}

    \begin{remark}[The connection with {\em ultra-fast diffusion}  on $\R^n$]\label{remark-cauchy}
In \cite{DP94, DP97}, the second author and M. del Pino studied the Cauchy problem of ultra-fast diffusion equations 
$u_t = \nabla \cdot(u^{m-1} \nabla u)$ on $\R^{n}$ for $m{\color{black}\le}-1$. In an attempt to find the fastest possible decay of initial data $u_0$ which guarantees a solvability of the equation on $t\in(0,T)$, some partial necessary or sufficient conditions had been found. As pointed earlier, the evolution of $H$ in the IMCF is similar to the ultra-fast diffusion equation of $m=-1$ and it shares similar features. Let us first summarize some of results when $m=-1$ from \cite{DP94,DP97}.
First, there is $C(n)$ so that if the Cauchy problem $u_t= \nabla \cdot (u^{-2} \nabla u)$ with $u(x,0)=u_0(x)\ge0$ has a solution for $t\in(0,T)$, then \[\limsup_{R\to \infty}\fr{1}{R^{n-1}} \int_{B_R}u_0dx \ge C\, T^{1/2}.\] There exist, however,  some $u_0(x)\ge0$ such that    \[\lim_{R\to\infty}\fr{1}{R^{n-1}} \int_{B_R} u_0 =C>0\] but  for which  {\em no solution exist}  with initial data $u_0$,  for any $T>0$. Such solutions are characterized by {\em a non-radial structure at spatial infinity}. Indeed, for initial data which is  bounded from below near infinity by {\em positive radial functions}  there is a necessary and sufficient for existence as follows:
 there is an explicit constant $E^*>0$ such that if the problem has a solution for $t\in(0,T)$, then \[\limsup_{R\to\infty} \left[{\fr1R} \left( \int_0^R \fr{ds}{w_ns^{n-1}} \int_{B_s} u_0\,  dx\right)\right]\ge E^*T^{1/2}.\] Moreover,  if $u_0$ is radially symmetric and locally bounded, \[\liminf_{R\to\infty} \left[{\fr1R} \left( \int_0^R \fr{ds}{w_ns^{n-1}} \int_{B_s} u_0 \, dx\right)\right]\ge E^*T^{1/2}\] guarantees an existence of a solution on $\R^n\times (0,T)$. For non-radial $u_0$, there is a similar condition in Theorem 1.3. \cite{DP97}.  Every result mentioned here is in some sense sharp when  explicit solutions $$v^T(x,t) = \fr{\sqrt{2(n-1) (T-t)_+}}{|x|} $$ are considered. These results explain partial conditions for non-existence and existence of solutions, but a complete description was missing. For the convex IMCF, however, Theorem \ref{thm-mainexistence} and \ref{thm-noextension} depict a fairly complete picture. This was possible by the geometric estimate Theorem \ref{thm-mainest3}. Note that this lower bound has the same decay of $v^T(x,t)$ above.  Instead of the integral operators used in \cite{DP97}, the asymptotic geometry of $M_0$ is used to provide the lower bound on $H$ in Theorem \ref{thm-mainest3}. It would be interesting  to see if a similar idea could be implemented in the theory of ultra-fast diffusion equation \eqref{eq-ultra}, with $m <0$.   
 
    \end{remark}
  

\appendix 
\section{Appendix}

\subsection{Strict convexity of solutions in space forms}\label{sec-strict}
Throughout this subsection, {\color{black}unless otherwise stated, we assume the solutions are smooth immersed n-dimensional possibly incomplete submanifods} in $(N^{n+1},\bar g)$ which is a {\em space form}  of sectional curvature $K\in \R$. This ambient space, in particular, includes Euclidean space, the sphere, or hyperbolic space. {\color{black}Since smooth solutions are strictly mean convex ($H>0$), this necessarily implies that the solutions are orientable.}  As before, we denote the outward unit normal which is opposite to the mean curvature vector by $\nu$, the norm of mean curvature by $H$, and the second fundamental form with respect to $-\nu$ by $h_{ij}$.  {\color{black}Moreover, we say a solution is convex if $h_{ij}$ is nonnegative definite everywhere. Note that the convex solution in this subsection is weaker notion than the convex solution in other sections which uses Definition \ref{def-convexity}. For instance, a $C^2$ hypersurface convex in the sense of Definition \ref{def-convexity} should necessarily be complete and embedded. }

Our aim is to prove Theorem \ref{thm-strict}, a strong minimum principle on $\lambda_1$.  However by looking at the evolution of the second fundamental form $h_{ij}$ given in  \eqref{eq-noneuch_ij}, 
it is not clear that  the convexity is preserved. To do so we need to use a  viscosity solution argument  and we  need the following lemma shown from  \cite{BCD}.  
\begin{lemma}[Lemma 5 in Section 4 \cite{BCD}]\label{lemma viscos}
Suppose that $\phi $ is a smooth function such that $\lambda_1 \ge \phi$ everywhere and $\lambda_1=\phi$ at $x=\bar p \in \Omega$. Let us choose an orthonormal frame so that \[h_{ij}=\lambda_i \delta_{ij}\text{ at }\bar p\in\Omega\quad\text{ with }\lambda_1=\lambda_2=\ldots=\lambda_\mu<\lambda_{\mu+1} \le \ldots \le \lambda_n.\] We denote $\mu\ge1$ by the multiplicity of $\lambda_1$.  Then at $\bar p$,    $\nabla_i h_{kl} =  \delta_{kl}\nabla_i\phi $ for $1\le k,l\le \mu$.  Moreover, $$\nabla_i \nabla_i \phi \le \nabla_i\nabla_i h_{11} - 2 \sum_{j>\mu} (\lambda_j-\lambda_1)^{-1}(\nabla_i h_{1j})^2 .$$
\end{lemma}

\begin{prop}\label{prop-visc} For $n\ge1$, let  $ F:\Omega \times(0,T)\rightarrow (N^{n+1},\bar g) $ be a smooth convex solution to the IMCF where $(N,\bar g)$ is a space form. Let  $\lambda_1$ 
denote the lowest eigenvalue of $h^i_{j}$. Then ${\ds u:= {\lambda_1}/{H}}$ is a viscosity supersolution to the  equation 
\be \label{eq-uvisc} \frac{\partial}{\partial t} u - \frac{1}{H^2} \Delta u + \frac{1}{H^3}\langle V , \nabla u \rangle +\left( \frac{W}{H^4}\right) u\ge 0
\ee where $V$ is a vector field, and $W$ is a scalar function such that $$|W|, |V| \le C(|\nabla H|,n)\quad\text { at each point.}$$   

\begin{proof} Using  equation \eqref{eq-noneuch_ij} in  Remark \ref{remark-diffambient}, we compute the evolution of  $h^i_{j}/H$:
\be\ba\label{eqn-hijH}
 \PP\fr{h^i_{j}}{H}&= 2\frac{|A|^2}{H^2}\frac{ h^i_{j}}{H}-2\frac{h^{ik}h_{kj}}{H^{2}} +\frac{2}{H^{4}}(\nabla_m H\nabla^mh^i_{j} -\nabla^{i} H \nabla_{j} H ) .\\
\ea\ee
Suppose a smooth function of space time, namely $\phi/H$, touches $\lambda_1/H$ from below at $(\bar p,\bar t)$.  At time $\bar t$ around $\bar p$, let us fix a time independent frame $\{ e_i\}$ using the metric $g({\color{black}\bar t})$ as in Lemma \ref{lemma viscos}.

Since ${\phi} \le {\lambda_1} \le {h^1_1}$ and they coincide at $(\bar p, \bar t)$, $\partial_t \phi \ge \partial_t h^1_1$ at $(\bar p, \bar t)$. 
At this point $(\bar p, \bar t)$ with the frame $\{e_i\}$, we use Lemma \ref{lemma viscos},  equation \eqref{eqn-hijH},  and the Codazzi   identity  $\nabla_{i}h_{jk} = \nabla_{j}h_{ik}$ to obtain\be\ba \label{eq-localconv1}
\square \frac{\phi}{H} &\ge \frac{\partial}{\partial t} \frac{h^1_1}{H} -\frac{1}{H^2} \Delta \frac{h_1^1}{H}+ \frac{2}{H^{3}}\sum_i\sum_{ j>\mu}(\lambda_j-\lambda_1)^{-1} |\nabla_i h_{1j}|^2 \\ &= \square\frac{ h^1_1 }{H}+ \frac{2}{H^{3}}\sum_{i\ge1,j>\mu} (\lambda_j-\lambda_1)^{-1} |\nabla_1 h_{ij}|^2 \\
&=  \frac{2\sum_j \lambda_j^2 - 2\lambda_1 \sum_j \lambda_j}{H^2}\frac{\lambda_1}{H}  + \frac{2}{H^{4}}\left[\nabla_m H\nabla_mh_{11} -|\nabla_{1} H|^2+H\sum_{ i\ge1,j>\mu}\frac{ |\nabla_1 h_{ij}|^2}{ \lambda_j-\lambda_1} \right]\\ 
&\ge \frac{2}{H^{4}}\left[\nabla_m H\nabla_m\phi -|\nabla_{1} H |^2+H\sum_{i\ge1,j>\mu} \frac{ |\nabla_1 h_{ij}|^2}{\lambda_j-\lambda_1} \right]. 
\ea\ee
The last inequality uses $\lambda_1 \sum_j \lambda_j\le \sum_j \lambda_j^2 =|A|^2$, {\color{black}which holds on convex (or more generally on mean convex) hypersurfaces.} 

Next, note that
	\be\label{eq-11}\ba
	\nabla_1H\nabla_1H=\sum_{i,j}\nabla_1h_{ii} \nabla_1 h_{jj} &=2\mu\nabla_1H \nabla_1\phi-\mu^2|\nabla_1\phi|^2 +	\sum_{ i>\mu,j>\mu}\nabla_1h_{ii}\nabla_1h_{jj}
	.\ea\ee 
Since ${\ds H \, \nabla \frac{\phi}{H}=  \nabla \phi-  \frac{\phi}{H} \, \nabla H}$, we have the following for each fixed unit direction $e_m$	\be\label{eq-m}\ba
	\nabla_m H \nabla_m \phi=H \nabla_m H \nabla_m \frac{\phi}{H} + \frac{\phi}{H}|\nabla_mH|^2.
	\ea\ee
	We first plug \eqref{eq-m} with $m=1$ into \eqref{eq-11} and then plug that into the last line of  \eqref{eq-localconv1} to obtain 
	\be\label{eq-phi/H}\ba
	\square \frac{\phi}{H} &\ge \frac{2}{H^4}\sum_{m>1}{ |\nabla_mH|^2}\frac{\phi}{H}	+\frac{2(1-2\mu)}{H^4}{ |\nabla_1H|^2} \frac{\phi}{H}
	\\ &\quad + \frac{2}{H^3} \sum_{m>1}{\nabla_mH\nabla_m\frac{\phi}{H}} + \frac{2(1-\mu)}{H^3} \nabla_1H \nabla_1 \frac{\phi}{H}+ \frac{2\mu^2}{H^3} \frac{|\nabla_1\phi|^2}{H}
	\\&\quad +\frac{2}{H^{4}}\left[ H\sum_{ i\ge1,j>\mu}(\lambda_j-\lambda_1)^{-1} |\nabla_1h_{ij}|^2- \sum_{ i>\mu,j>\mu}\nabla_1h_{ii} \nabla_1h_{jj}\right].\ea\ee
	 We now use   the convexity,  $\lambda_1\ge0$, in the proof of the following claim.
\begin{claim}\label{claim-strict}$\left[ H\sum_{ i\ge1,j>\mu}(\lambda_j-\lambda_1)^{-1} |\nabla_1h_{ij}|^2- \sum_{ i>\mu,j>\mu}\nabla_1h_{ii} \nabla_1h_{jj}\right]\ge0$ on $\{\lambda_1\ge0\}$.
\end{claim} Assuming that  the claim is true,  then by  taking  away the  good term ${\ds \frac{2\mu^2}{H^3} \frac{|\nabla_1\phi|^2}{H}}$ in \eqref{eq-phi/H}, 
we easily conclude that  \eqref{eq-uvisc} holds by choosing  a vector filed $V$ and a scalar function $W$ as a function of $\nabla H$ accordingly. Thus it remains to show the claim. 
\begin{proof}[Proof of Claim \ref{claim-strict}] Since $\lambda_1\ge0$, $H=\sum_{l\ge1} \lambda_l \ge \sum_{l>\mu} \lambda_l$, the claim follows by:
	\be\ba
	H\sum_{i\ge1,j>\mu}(\lambda_j-\lambda_1)^{-1} |\nabla_1h_{ij}|^2&\ge \sum_{l>\mu}\lambda_l \sum_{i>\mu}\lambda_i^{-1}|\nabla_1 h_{ii}|^2 = \sum_{i>\mu,j>\mu}\lambda_j \lambda_i^{-1}|\nabla_1 h_{ii}|^2 \\ & = \sum_{i>\mu,j>\mu}\frac{\lambda_j \lambda_i^{-1}|\nabla_1 h_{ii}|^2 +\lambda_i \lambda_j^{-1} |\nabla_1 h_{jj}|^2}{2} \\ & \ge  \sum_{i>\mu,j>\mu} \nabla_1 h_{ii} \nabla_1 h_{jj}.
\ea\ee
\end{proof}
\end{proof}
\end{prop}
Now, let $M_t\subset N^{n+1}$ be a smooth  complete convex solution for $t>0$, which could be either compact or non-compact. One  expects $M_t$ to be  strictly convex, that is  to have $\lambda_1>0$ for $t>0$. 
%
%
Indeed, this follows  easily by  Proposition \ref{prop-visc} and the strong minimum principle for nonnegative supersolutions which is a consequence of the weak Harnack inequality for nonnegative viscosity super solutions to (locally) uniformly parabolic equations. (See Chapter 4 in \cite{W}).
\begin{theorem} \label{thm-strict}Suppose $F:M^n\times (0,T)\rightarrow(N^{n+1},\bar g)$ is a smooth complete convex solution to the IMCF with $H>0$ where $(N^{n+1},\bar g)$ is a space form. If $\lambda_1(p_0,t_0)=0$ at some $(p_0,t_0)$ with  $0<t_0<T$, then  $\lambda_1=0  $ on $M^n\times (0,t_0]$.
\begin{proof}
Since solution is smooth, $|H|$, $|\na H|$, and $|H^{-1}|=|\partial_t F|$ are locally bounded. Therefore, $\lambda _1$ is a nonnegative supersolution to equation \eqref{eq-uvisc} which is locally uniformly parabolic with bounded coefficients.  We can apply strong minimum principle on {a sequence $\{ \Omega_k \}$ of expanding domains  containing $(p_0, t_0)$ such that $M^n = \cup_k \Omega_k$} and conclude that the theorem holds. 
\end{proof}
\end{theorem}

\begin{cor}\label{cor-strict}
Let $F:M^n\times (0,T)\rightarrow \R^{n+1}$ be a smooth complete convex solution to the IMCF. If $\mathcal{H}^{n}(\nu[M_{t_0}]) >0$ at $t_0\in (0,T)$, then $M_{t_0}$ is strictly convex.
\begin{proof} If it is not, Theorem \ref{thm-strict} implies $\lambda_1\equiv 0$ for all $M^n \times(0,t_0]$. {\color{black} The Gauss map $\nu : M_{t_0} \to\mathbb{S}^n$ is Lipschitz. Thus the area formula implies 
\be \int_{M_{t_0}} K d\mu=\int_{\mathbb{S}^n} \mathcal{H}^0( \nu^{-1}(\{z\}))d\mathcal{H}^n (z)\ge \mathcal{H}^{n}(\nu[M_{t_0}]). \ee 
Since $K\equiv 0$,} this is a contradiction and proves the assertion.  
\end{proof} 
\end{cor}


\begin{remark}[Strict convexity of compact solutions] \label{remark-conv}  We may use Theorem \ref{thm-strict} to show the initial strict convexity of smooth compact solution is preserved. {\color{black}Let $M_t$ be the smooth IMCF running from smooth compact immersed hypersurface $M_0$ with positive $h_{ij}$.}  By considering the first time $\lambda_1$ becomes zero at some point, Theorem \ref{thm-strict} implies that there is no such time as long as the smooth solution exists and this proves $h_{ij}$ is positive for $M_t$. {\color{black}Furthermore, if $M_t$ is a flow in $\mathbb{R}^{n+1}$, then it is the boundary of a compact convex set with interior by Hadamard \cite{Hadamard}, showing $M_t$ is convex in the sense of Definition \ref{def-convexity}.}
\end{remark}

\subsection{Speed estimate for closed star-shaped solutions}\label{appendix-starshaped}

The goal is this section is to give an alternative proof  of Theorem 1.1 in \cite{HI} which will be based  on the maximum principle. The theorem holds in  any dimension $n \geq 1$. 

\begin{thm}[Theorem 1.1 in \cite{HI}]\label{thm-equivHI} Let $F : M^n \times [0,T]\to \R^{n+1}$ be a smooth closed star-shaped solution to \eqref{eqn-IMCF} such that $M_0:=F_0(M^n)$
satisfies 
\be\label{eqn-star} 0 < R_1 \leq \la F,\nu \ra \leq R_2.
\ee
Then,  there is a constant $C_n>0$ depending only on $n$ such that 
\be\label{eqn-main}
 \fr{1}{H}\le C_n \, \left (  \fr{R_2}{R_1} \right ) \, \Bigl ( 1+ \fr{1}{t^{1/2}} \Bigr ) \,R_2 \, e^{\frac {t}n}\ee
 holds everywhere on $M^n\times [0,T].$

\end{thm} 

\begin{proof} 
Since  $M_0$ satisfies \eqref{eqn-star}, by  Proposition 1.3 in \cite{HI},  we have 
\be\label{eqn-star2} 
R_1 \leq  R_1 \, e^{\fr{t}{n}} \le \la F, \nu\ra \leq |F|   \le R_2 \,  e^{\fr{t}{n}} \ee
for all $0 < t < +\infty$. 
Let us denote  $w:= \langle F, \n\rangle ^{-1}$ and we will consider a function $$Q:= \fr{ \varphi ^{1-\e}( w ) \, e^{\gamma |F|^2}}{H}$$ for some function $\varphi:=\varphi(w)$, constants $\gamma>0$ and $\e\in(0,1)$ which will be chosen shortly. 

{\color{black} By (i) in Lemma \ref{lem-23},} \[ \ba \PP \ln e^{|F|^2} =\PP |F|^2= -\fr{2n}{H^2} +\fr{4}{Hw}.\ea\] 
Moreover, by $(7)$ in Lemma \ref{lem-HI1}  and Lemma \ref{lem-com1} with $\beta=-1$, \[\ba \PP w&= - \fr{|A|^2}{H^2} w-\fr{2}{wH^2} |\na w|^2  \ea\] and hence, on $\{\varphi \neq 0\}$,
\[\PP \ln \varphi = \fr{\pp \varphi}{\varphi} + \fr{1}{H^2}\fr{|\nabla \varphi|^2}{\varphi^2}= -\fr{|A|^2}{H^2} \fr{\varphi' w}{\varphi} -\fr{|\nabla w|^2 }{H^2} \, \left (2 \, \fr{\varphi'}{w\varphi} + \fr{\varphi''}{\varphi}-\fr{\varphi'^2}{\varphi^2} \right ).\] 
Inspired by the choice of $\varphi$ in the well known  interior curvature estimate by Ecker and Huisken in \cite{EH} (see also \cite{CNS}), we define \be \label{eqn-functions}
\varphi (s):=\left (\fr{s}{2R_1^{-1}-s}\right).
 \end{equation} 
For this $\varphi:=\varphi(w)$, under the notation $\varphi'= \varphi'(w)$ and $\varphi''=\varphi''(w)$, a direct computation yields $$\fr{\varphi'w}{\varphi} =- \Bigl( \fr{2}{2-wR_1}\Bigr )\quad\text{and}\quad2 \, \fr{\varphi'}{w{\color{black}\varphi}} + \fr{\varphi''}{\varphi}-\fr{\varphi'^2}{\varphi^2}  = \fr{\varphi'^2}{\varphi^2}.$$
Lemma \ref{lem-pg} and the computations above imply 
\be\ba \label{eqn-all1} \PP \ln Q=&  \left[\fr{|A|^2}{H^2}+\fr{1}{H^2}  \fr{|\nabla H^{-1}|^2}{H^{-2}}\right] 
+ \gamma \left[\fr{4}{H  w} -\fr{2n}{H^2}\right]- (1-\e)\left [\fr{|A|^2}{H^2} \fr{\varphi' w}{\varphi}+\fr{1}{H^2} \fr{|\nabla \varphi|^2}{\varphi^2}\right] \\=&  -\Bigl( \fr{wR_1-2\e}{2-wR_1}\Bigr ) \, 
\fr{|A|^2}{H^2} \,
+ \, \Bigl(\gamma\fr{4}{H  w} -\gamma \fr{2n-4  \e^{-1}   \, \gamma \, |F|^2 \, \big |\nabla\, |F| \big |^2 }{H^2}\Bigr)\\& -\fr{1}{H^2} \left [ (1-\e)\fr{|\nabla \varphi|^2}{\varphi^2}- \fr{|\nabla H^{-1}|^2}{H^{-2}}  
+   \e^{-1}\gamma^2 \fr {|\nabla e^{ |F|^2}|^2}{| e^{|F|^2}|^2}\right]   .\ea\ee
Note that we have added and subtracted the term ${\frac{ \e^{-1} \gamma^2 \,\big  |\nabla |F|^2 \big |^2  }{H^2}}$ in the last equality.
At a nonzero critical point of $Q$, \[0=\fr{\nabla Q}{Q} =(1-\e) \fr{\nabla \varphi}{\varphi} +\gamma \fr{\nabla e^{|F|^2}}{e^{|F|^2}}+{\color{black}\fr{\nabla H^{-1}}{H^{-1}}}, \] and thus 
\[\ba \left|\fr{\nabla H^{-1}}{H^{-1}}\right|^2 = \left|(1-\e)\fr{\nabla \varphi}{\varphi} + \gamma\fr{\nabla e^{ |F|^2}}{e^{ |F|^2}}\right|^2 &= (1-\e)^2 \left|\fr{\nabla \varphi}{\varphi}\right|^2 + 2(1-\e)\gamma \left\la \fr{\nabla \varphi}{\varphi},\fr{\nabla e^{ |F|^2}}{e^{ |F|^2}}\right\ra+\gamma^2 \left |\fr {\nabla e^{|F|^2} } {e^{|F|^2} } \right|^2  \\&\le ((1-\e)^2+ {\e}({1-\e})) \left|\fr{\nabla \varphi}{\varphi}\right|^2 + (1+\fr{1-\e}{\e})\gamma^2 \left |\fr {\nabla e^{|F|^2} } {e^{|F|^2} } \right|^2 \\ &=  (1-\e)\fr{|\nabla \varphi|^2}{\varphi^2} +  \e^{-1}\gamma^2  \fr {|\nabla e^{|F|^2}|^2}{ e^{2 |F|^2}}. \ea\]
For a given $T >0$, note that $\fr{R_1}{R_2e^{\frac{T}{n}}} \le wR_1\le 1 $. It remains to choose $\e$ and $\gamma$. The choice $\e   :=  \fr{R_1}{2R_2e^{\frac{T}{n}}}$ makes the first term on RHS of the second equality in \eqref{eqn-all1} nonpositive. 
Next, choose $\gamma := \fr{\e}{4n}\fr{1}{(R_2e^{\frac{T}{n}})^2} >0$ so that $4\e^{-1}\gamma {|F|^2}  \le n$ on $M_t$ for $t \in [0,T]$.
Combining the choices and estimates, at a nonzero spatial critical point of $Q$,
\be\label{eqn-Q20}
\PP \ln Q  = \fr{\pp Q}{Q} +\fr{|\nabla Q|^2}{Q^2}\leq \gamma  \, \left ( -\fr{n }{H^2}+\fr{4}{H \, w} \right ). 
\ee

We will now apply the maximum principle on $\hat Q:= t  Q$.  Suppose that nonzero maximum of $\hat Q$ on $M ^n \times [0,T]$ occurs at the point $(p_0,t_0)$, which necessarily implies $t_0>0$.   At this point, \eqref{eqn-Q20} implies 
\be\label{eq-maxpoint1} 0 \le \PP  \ln \hat Q  \le \gamma \, \left( -\fr{n }{H^2}+\fr{4}{H \, w} \right)  + \fr{1}{t_0} 
\le \gamma \, \left( -\fr{n}{2H^2} +\fr{8}{n}R_2^2e^{2\frac{T}{n}}\right) + \fr{1}{t_0}\ee
where the second inequality comes from  $$ \frac {4}{H\,w} \leq\fr{8}{nw^2}+  \frac n{2 H^2} \leq \fr{8}{n}R^2_2e^{2\frac{T}{n}} +  \frac n{2 H^2}.$$  The rest is a standard argument shown in the proof of Theorem 3.1 \cite{EH}. By the choices of $\e$, $\gamma$, bounds \eqref{eqn-star2}  and \[ \fr{R_1}{2R_2 e^{T/n}} \le\varphi((R_2e^{T/n})^{-1}) \le \varphi(w) \le \varphi(R_1^{-1})=1,   \] 
we proceed and obtain, for every $(p,t)\in M^n\times (0,T]$,
\be\fr{1}{H^2}(p,t)\le C_n  \left(\fr{R_2}{R_1}e^{\fr{T}{n}}\right)^{2-\e}(R_2e^{\fr{T}{n}})^2 \left(1+\frac{1}{t}\right). \ee

Now  for  time $t>1$, we can alway apply this estimate starting at  time $t-1$. Inequality \eqref{eqn-star2} implies that the ratio between star-shapedness bounds from above and below remains  unchanged over time. This way we can replace $({R_2}e^{\fr{T}{n}}/R_1)^{2-\e}$ in the above estimate by $({R_2}/R_1)^{2-\e}$ after possibly enlarging the  constant $C_n$. Since $( {R_2}/ {R_1})^{2-\e} \le (R_2/R_1)^2$, the  theorem follows.
\end{proof}

\subsection{Smooth approximation of convex hypersurfaces}\label{appendix-approximation}

In this appendix, unless it is stated otherwise, {\em  convergence}  of compact convex sets (or their boundary hypersurfaces) means {\em the convergence in the Hausdorff metric}, defined as 

\be\label{eq-Hausdorff}d_H(A,B) := \max (\sup_{x\in A} \,\inf_{y\in B} \Vert x-y\Vert , \sup_{x\in A}\, \inf_{y\in B} \Vert x-y\Vert  ).\ee

In case where  $A$, $B$ are compact convex sets, it is known that $d_H(A,B)= d_H(\p A, \p B)$ (Lemma 1.8.1 \cite{Rolf}).  For a compact convex set $\hat M\subset \mathbb{R}^{n+1}$, $\hat M ^\delta$ denotes the $\delta$-envelope \[\hat M^\delta := \{ x\in \mathbb{R}^{n+1}\,:\,\dist(x, \hat M) \le \delta \}=\hat M + \delta \overline{B_0(1)}. \] Here the $\dist(x,\hat M)$ is measured in {\color{black}Euclidean} distance.

\begin{lemma}\label{lem-euclidapprox}

{\color{black} For $n\ge1$,} let $\hat M\subset \mathbb{R}^{n+1}$ be a compact convex set with non-empty  interior. Suppose $0\in \text{int}(\hat M)$. Then there is a sequence compact convex sets $\hat M_k$ with smooth strictly convex boundaries  such that \[\hat M^{de}_k:= (1+k^{-1})\hat M_k \text{ is strictly decreasing},\quad \hat M^{in}_k:= (1-k^{-1}) \hat M_k \text{ is strictly increasing, }\]
and both {\color{black}converge} to $\hat M$ as $k\to \infty$. Here, we say that $\hat \Sigma^{de}_k$ ($\hat \Sigma^{in}_k$) is strictly decreasing (strictly increasing) if {\color{black}$\hat \Sigma^{de}_{k+1} \subset \text{int} \, (\hat \Sigma^{de}_k)$ ($\hat \Sigma^{in}_{k} \subset \text{int} \, (\hat \Sigma^{in}_{k+1})$}, respectively). 
\end{lemma}
\begin{proof} 
By Theorem 3.4.1  in \cite{Rolf} and its immediate following discussion, there is a sequence of compact convex sets $\hat M_k$ with non-empty interior and smooth strictly convex boundaries such that $d_H(\hat M,\hat M_k) \to 0$.  Note that $(1+k^{-1})\hat M$ and $(1-k^{-1})\hat M$ are strictly monotone sets. Hence, using a diagonal argument, we may choose a subsequence of $\hat M_k$ so that $(1+k^{-1})\hat M_k$ and $(1-k^{-1})\hat M_k$ are strictly monotone as well. 
\end{proof}

Let $\hat \Gamma$ be a compact convex set in the upper open hemisphere $\mathbb{S}^n\cap \{x_{n+1}>0\}=:\mathbb{S}^n_+$. Since $\mathcal{C}\hat \Gamma$ is convex in $\mathbb{R}^{n+1}$, if we define $\hat \Omega \times \{1\} := \mathcal{C}\hat \Gamma \cap \{x_{n+1}=1\}$, then $\hat\Omega$ is a compact convex set in $\mathbb{R}^n$. $\hat \Gamma$ and $\hat \Omega$ are related by $\hat \Gamma = \varphi (\hat \Omega)$ using the {\color{black}gnomonic} projection $(x,1) \in \mathbb{R}^{n+1}\mapsto \varphi (x)\in\mathbb{S}^n_+$  \be\label{eq-varphi}\varphi(x) = \frac{1}{(1+|x|^2)^{1/2}}(x,1) \quad\text{ for }x\in\mathbb{R}^n.\ee  Since $(\mathbb{R}^n,\varphi^* g_{\mathbb{S}^n})$ and $(\mathbb{S}^n_+, g_{\mathbb{S}^n})$ are isometric, we have $|\p \hat \Gamma | = |\p \hat \Omega|_{\varphi^* g_{\mathbb{S}^n}}$.  {\color{black}Here, the surface area measure $|\cdot|_{\varphi^* g_{\mathbb{S}^n}}$ is $(n-1)$-Hausdorff measure on metric space $(\mathbb{R}^n,d_{\varphi^* g_{\mathbb{S}^n}})$, induced from Riemannian structure $(\mathbb{R}^n,\varphi^* g_{\mathbb{S}^n})$. Since $\p \hat \Omega$ is $(n-1)$-rectifiable, we will use the area formula and avoid using the metric $d_{\varphi^* g_{\mathbb{S}^n}}$ in actual computation of $|\p \hat \Omega|_{\varphi^* g_{\mathbb{S}^n}}$.}

\begin{lemma}\label{lem-b2} {\color{black} For $n\ge1$,} assume that a sequence of compact convex sets $\hat M_k \subset \mathbb{R}^{n+1}$ converges to a compact convex set $\hat M$ with non-empty interior. Then $\lim_{k \to \infty} |M_k| = |M|$. Similarly, if a sequence of compact convex sets $\hat \Gamma_k \subset \mathbb{S}^n_+$ converges to a compact convex set with non-empty interior $\hat \Gamma \subset  \mathbb{S}^n_+$, then $\lim_{k \to \infty} |\Gamma_k| = |\Gamma|$.   

\begin{proof}
Our proof is a modification of the proof of Theorem 4.2.3 in \cite{Rolf}. For the first part, we may assume $0 \in \text{int}(\hat M_k)$ for all $k$. Let $\rho (\hat M_k, \cdot)$ be a spherical parametrization of $ M_k$ around $0$, meaning that $\rho(\hat M_k,y)\, y$ for $y\in \mathbb{S}^n$ is a point in $M_k$.  Let   $\nu(\hat M_k,y)$  denote an arbitrary outer unit normal (in the sense of supporting normal) of $\hat M_k$ at $\rho (\hat M_k, y) y$. The normal $\nu(\hat M_k,y)$ is unique for $\mathcal{H}^{n}$-almost all $y\in \mathbb{S}^{{\color{black}n}}$ (Theorem 2.2.5 in \cite{Rolf}). If $\hat M_k$ converges 
to $\hat M$ as $k\to \infty$, then $\rho(\hat M_k,\cdot) \to \rho (\hat M,\cdot)$ everywhere.
Moreover, $\nu(\hat M_k,\cdot)\to \nu(\hat M, \cdot)$ $\mathcal{H}^{n}$-almost everywhere,  otherwise  if  $\nu(\hat M_{i_k},y')\to \nu'$ for some $\nu'$ as $i_k\to\infty$, this  would implies that $M$ has a supporting hyperplane $\{\la x-\rho(\hat M, y')y' , \nu' \ra =0\}$ at $\rho(\hat M,y')y'$, but on the other hand the outer normal of $M$ uniquely exists  $\mathcal{H}^{n}$-almost everywhere. Finally, since $\hat M_k$ contains the origin in its interior, there is a uniform $\delta>0$ such that $\la y ,\nu(\hat M_k, y)\ra \ge \delta$ for all $k$ and $y$.

Let $\psi:U\subset \mathbb{R}^n \to \mathbb{S}^n$, $z=(z^1,\ldots,z^n)\mapsto \psi(z)=y$, be a smooth local coordinate chart of $\mathbb{S}^n$ and $g_{ij}$ be {\color{black}the metric $ g_{\mathbb{S}^n}$ on $U$.} Note that $\rho(\hat M_k,\psi (\cdot))$ is a Lipschitz function. At each point $y$ where the  function  $\rho_k(y):=\rho(\hat M_k,y)$ is differentiable, one can directly compute that \[\nu(\hat M_k,y) = \frac{\rho_k y -g^{ij}\fr{\p \rho_k}{\p z^i} \fr{\p {\color{black}y}}{\p z^j}}{\sqrt{\rho_k^2 + \Vert d\rho_k\Vert _{g_{\mathbb{S}^n}}^2}} .\] Thus from the convergence of $\nu(\hat M_k,\cdot)$ and the lower bound $\la y ,\nu(\hat M_k, y)\ra \ge \delta$, if $\rho_{\infty}(y):=\rho(\hat M, y)$, then we have $\left|\frac{\p \rho_k}{\p z^i} \right|\le C_\delta$ and $\frac{\p \rho_k }{\p z^i}\to \frac{\p \rho_\infty}{\p z^i} $ almost everywhere for all $i=1,\ldots,n$.

Denote by $f_k$ and $f_\infty$ $: U \to\mathbb{R}^{n+1}$ the functions $f_k(z) :=\rho(\hat M_{\color{black}k},\psi (z))\psi(z)$ and $f_\infty(z) := \rho(\hat M, \psi(z))\psi(z)$. By the area formula, we have  $|f_k(U)|= \int_U J_{f_k}(z) dz^n $,  where 
\[\ba J_{f_k}(z) &= \sqrt{\det \left[ \left\la \fr{\p f_k}{\p z^i}(z), \fr{\p f_{k}}{\p z^j}(z)\right\ra\right]}= \sqrt{\det \left[\frac{\p \rho_k}{\p z^i}\frac{\p \rho_k}{\p z^j}+\rho_k^2 g_{ij} \right]}\\&=\sqrt {\rho_k^{2n-2}(\rho_k^2 +\Vert d\rho_k\Vert_{g_{\mathbb{S}^n}} ^2 )}\sqrt{\det g_{ij}}= \rho_k^{n-1} \sqrt{\rho_k^2+ \Vert d\rho_k\Vert ^2_{g_{\mathbb{S}^n}}} \sqrt {\det g_{ij}}\ea \]
holds almost everywhere. 
The  Lebesgue dominated convergence theorem, then yields that 
\[|f_k(U)| = \int _U J_{f_k}(z) dz^n \to \int_U \rho_\infty^{n-1}\sqrt{\rho_\infty^2+{ \Vert d\rho_\infty\Vert _{g_{\mathbb{S}^n}}^2}} \sqrt {\det g_{ij}} = |f_\infty(U)|, \quad  \mbox{as}\,\, k \to \infty. \]
Using a standard partition of unity argument, we conclude that $|M_k| \to |M|$ as $k\to \infty$, which concludes the proof of the first assertion of the lemma.

\smallskip
We will now  prove the second assertion of the lemma.  Note that the preimages of $\hat \Gamma_k$ and $\hat \Gamma$ under $\varphi$ are compact convex sets in $\mathbb{R}^n$. On a given compact set, the metrics induced by $(\mathbb{R}^n, \varphi^* g_{\mathbb{S}^n})$ and  $(\mathbb{R}^n, g_{\mathbb{R}^n})$ are equivalent. i.e. one is less than a constant multiple of the other and the constant depends on the compact set.  We may assume that there is some $p\in \mathbb{R}^n$ such that $\hat M_k:=\varphi^{-1}(\hat\Gamma_k)-p$ and $\hat M:= \varphi^{-1}(\hat \Gamma)-p$ contain the origin in their interiors and $d_H(\hat M_k, M) \to 0$ as $k\to \infty$. 

Defining  $\rho_k$, $\rho_\infty$ $:\mathbb{S}^{n-1}\to \mathbb{R}$, a smooth local chart $\psi: U\subset \mathbb{R}^{n-1}\to \mathbb{S}^{n-1}$, and  the functions $f_k$, $f_\infty$ $: U \to \mathbb{R}^n $ similarly as in the previous  case (note that the dimension is $1$ less than the dimension in the previous case), we get the convergence of $\rho_k$ to $\rho_\infty$ with positive uniform upper and lower bounds and the convergence of $\frac{\p \rho_k}{\p z^i}$ to $\frac{\p \rho_\infty}{\p z^i}$ with uniform bound on their absolute value. 
With $\tilde \varphi(x) := \varphi(x+p)$, the area formula says  that $|\varphi(f_k(U)+p) |= \int_U J_{\tilde\varphi \circ f_k}(z)dz^n$,  where 
\[J_{\tilde\varphi \circ f_k}(z) = \sqrt{\det \left[ \left\la \frac{\p\tilde\varphi}{\p x^\alpha}(f_k(z)) \frac{\p f_k^\alpha}{\p z^i}(z) ,\frac{\p\tilde\varphi}{\p x^\alpha} (f_k(z))\frac{\p f_k^\alpha}{\p z^j}(z)\right\ra\right]} .\]
Note that $\varphi$ is a smooth function, which in particular has bounded higher order derivatives  on each compact domain. Therefore  the Lebesgue dominated convergence theorem yields that  
$$|\varphi(f_k(U)+p)| \to |\varphi(f_\infty(U)+p)|, \quad \mbox{as}\,\, k \to \infty$$  and a partition of unity can be used to show 
\[|\varphi(M_k +p)|=|\Gamma_k| \to|\varphi(M+p)| =|\Gamma|.\]

\end{proof}

\end{lemma}

The following is a well known lemma and a stronger statement than this also holds, but we provide a proof of this simple version for the completeness of our work. 

\begin{lemma}\label{lem-outermin}
{\color{black} For $n\ge1$,} a  convex hypersurface $M=\p\hat M$ in $\mathbb{R}^{n+1}$ has the outer area minimizing property among convex hypersurfaces,  i.e. if $\hat M \subset \hat M'$ then $| M|\le |M'|$.  Moreover, when $M$ is smooth strictly convex, then the  property is strict in the sense that  equality holds if and only if $M=M'$.

 \begin{proof}
The proof uses a standard calibration argument. Assume $0\in\text{int}(\hat M)$. Assume $M$ is smooth and strictly convex. Then 
$\lambda M$, $\lambda\ge1$, gives a foliation of  smooth strictly convex hypersurfaces. Let $\hat M'$ be a set containing $\hat M$.  The foliation gives a smooth vector field consisting  of the outer normal vectors $\nu$ of $\{\lambda M\}_{\lambda \ge1}$. {\color{black}If we denote the unit normal on $\p(\hat M' \setminus \hat M) $ pointing from $\hat M' \setminus \hat M$ by $\nu '$, then} the divergence theorem implies 
{\color{black} \[0\le \int_{\hat M' \setminus \hat M} H = \int_{\hat M' \setminus \hat M} \text{div }\nu  = \int_{M'} \la \nu,\nu' \ra  d A + \int_M \la \nu,\nu '\ra  d A=\int_{M'} \la \nu, \nu '\ra  dA - \int_M dA    \]}and hence  
 \[|M|=\int_M dA \le \int_{M'}{\color{black} \la \nu, \nu' \ra} d A \le \int_{M'} dA =|M'|. \] (The divergence theorem can be applied  for a set with rough boundary when the boundary consists of convex hypersurfaces. One could also avoid doing this by approximating $\hat M'$ from outside using Lemma \ref{lem-euclidapprox} and Lemma \ref{lem-b2}).  The strict outer area minimizing is a consequence from the fact $H>0$ on $\hat M' \setminus \hat M$.

For a general convex $M=\p \hat M$, we consider smooth approximation from inside, say $\hat M^{in}_k$ which was shown to exist in  Lemma \ref{lem-euclidapprox}. By the first case, $|M^{in}_k| \le |M'|$. Lemma \ref{lem-b2} implies that $|M^{in}_k| \to |M|$ as $k\to \infty$ and this finishes the proof.
\end{proof}

\end{lemma}

We do have a similar result for convex hypersurfaces in $\mathbb{S}^n$.

\begin{lemma} \label{lem-outersphere} {\color{black}For $n\ge 2$,} let $\hat \Gamma\subset \mathbb{S}^{n}\cap \{x_{n+1}>0\}= \mathbb{S}^n_+ $ be a compact convex set with  non-empty interior. Then $\hat \Gamma$ can be approximated from inside (and outside) by a  strictly monotone    sequence  of compact sets in $\mathbb{S}^n_+$ with smooth strictly convex boundaries.  For any sequence of compact convex sets $\hat \Gamma_k$ converging to $\hat \Gamma$, 
we have $|\Gamma_k|\to |\Gamma|$ as $k\to \infty$. Moreover, $\hat \Gamma$ satisfies the outer area minimizing property on $\mathbb{S}^n_+$.
That is, if  $\hat\Gamma \subset \hat \Gamma'\subset\mathbb{S}^n_+ $ and $\hat \Gamma'$ is convex, then   $|\Gamma|\le |\Gamma'|$, and 
if $\Gamma$ is smooth strictly convex, then $|\Gamma|=|\Gamma'|$ if and only if $\Gamma=\Gamma'$. 
\end{lemma} 

\begin{proof}
Let $\hat \Omega := \varphi^{-1}(\hat \Gamma)\subset\mathbb{R}^n$. Then by Lemma \ref{lem-euclidapprox}, $\hat \Omega$ could be approximated from inside (and outside) by  strictly monotone  sequence of compact sets with smooth strictly convex boundaries. The images of these sequences of sets under $\varphi$ give the desired approximating sequences. The convergence of area is shown in Lemma \ref{lem-b2}.

The proof of the second part is similar to the proof of Lemma \ref{lem-outermin}. Suppose first $\Gamma=\p \hat \Gamma$ is smooth and $\hat \Gamma \subset \hat \Gamma' \subset \mathbb{S}^n \cap \{x_{n+1}>0\}$. Fix $p\in \hat \Omega=\varphi^{-1}(\hat \Gamma)$ and consider the foliation $\{\lambda(\Omega -p) +p\}_{\lambda\ge1}$. Then the  image of this  foliation under $\varphi$, that is  \[\varphi (\lambda(\Omega -p) +p) \subset \mathbb{S}^n, \text{ for }\lambda\ge1,\]
 gives a foliation of the region $\mathbb{S}^n_+ - \text{int}(\hat \Gamma)$ by smooth convex hypersurfaces in $\mathbb{S}^n$. By the same calibration argument, we obtain $|\Gamma'|\ge |\Gamma|$. In the non-smooth case  we approximate $\Gamma$ from inside by smooth sets  and apply Lemma \ref{lem-b2}.

\end{proof}

The next approximation   lemma concerns   with the case where  $\hat \Gamma\subset \mathbb{S}^n_+$ has empty interior. 

 \begin{lemma}\label{lem-b5} {\color{black} For $n\ge2$}, suppose $\hat \Gamma$ in $\mathbb{S}^{n}_+$ is a compact convex set which has empty  interior in $\mathbb{S}^n$. Then there is $\{\hat\Gamma_k\}$ a sequence of compact convex sets with non-empty interior and smooth strictly convex boundaries which strictly decreases to $\hat \Gamma$. For any such sequence $\hat \Gamma_k$, $|\Gamma_k|=|\p\hat \Gamma_k|$ decreases to $P(\hat \Gamma)$ as $k\to\infty$. Here $P(\hat \Gamma)$ is defined by \eqref{eq-perimeter}.  
  \begin{proof} 
$\hat \Omega= \varphi^{-1}( \hat \Gamma)$ has empty  interior in $\mathbb{R}^n$. 
Consider the set $\hat \Omega^{\delta}:=$ $\delta$-envelope of $\hat \Omega$ in $(\mathbb{R}^n, g_{\mathbb{R}^n} )$. Then $\hat \Omega^\delta $ is a compact convex set with non-empty interior, implying that $\varphi (\hat \Omega^\delta)$ is a closed convex set with non-empty interior in $\mathbb{S}^n$. By a diagonal argument applied to  the approximations of $\varphi(\hat \Omega^{1/k})$, which is similar to the proof of Lemma \ref{lem-euclidapprox}, we obtain  the existence of a strictly decreasing approximation. 

\smallskip   
Let us now prove the convergence $|\Gamma_k| \to P(\hat \Gamma)$, as $k \to \infty$.   If a convex set $\hat \Omega\subset \mathbb{R}^{n}$ has 
empty  interior, it is contained in a hyperplane of $\mathbb{R}^{n}$.  Since a rotation is an isometry of $(\mathbb{R}^n, \varphi^* g_{\mathbb{S}^n})$, we may assume that $\hat\Omega \subset \{x_n= l\}$,  for some $l>0$. Let  $\hat \Sigma \subset\mathbb{R}^{n-1}$ denote  the projection of $\hat \Omega$ to $\{x_n=0\}=\mathbb{R}^{n-1}$ and  $\hat \Sigma^\delta$ 
denote  the $\delta$-envelope of $\hat\Sigma$ in $\mathbb{R}^{n-1}$ with respect to the standard Euclidean metric. Observe that $\hat \Sigma^\delta \times \{l\} = \hat \Omega^\delta \cap \{x_{n}=l\}$. 
Moreover,   $\hat \Omega^\delta \subset\hat  \Sigma^\delta \times [l-\delta,l+\delta]$. The outer area minimizing property (Lemma \ref{lem-outersphere}) implies that 
  \[\ba |\p \hat \Omega^\delta |_{\varphi^* g_{\mathbb{S}^n}}& \le|\p (\hat\Sigma^\delta \times [-\delta,\delta] {\color{black})}|_{\varphi^* g_{\mathbb{S}^n}}\\&= |\p \hat\Sigma^\delta\times [l-\delta,l+\delta]|_{\varphi^* g_{\mathbb{S}^n}} +   |\hat \Sigma^\delta \times \{l-\delta\}|_{\varphi^* g_{\mathbb{S}^n}}  + |\hat \Sigma^\delta \times \{l+\delta\}|_{\varphi^* g_{\mathbb{S}^n}}    .\ea\]   It is clear that as $\delta \to 0$,  the first term in the last line is of order $O(\delta)$. Since $\hat \Sigma^\delta$ decreases to $\hat \Sigma$, together with the smoothness of $\varphi$, we conclude that each of remaining two terms converges to $|\hat\Sigma\times \{l\}|_{\varphi^* g_{\mathbb{S}^n}} $. This shows $\limsup_{\delta \to 0}|\p \hat \Omega^\delta |_{\varphi^* g_{\mathbb{S}^n}} \le P(\hat\Gamma).$ Next, $\p \hat \Omega^\delta$ contains $\hat \Sigma\times \{l-\delta\}$ and $\hat\Sigma\times \{l+\delta\}$,  implying that  \[  |\hat \Sigma \times \{l-\delta\}|_{\varphi^* g_{\mathbb{S}^n}} +|\hat \Sigma \times \{l+\delta\}|_{\varphi^* g_{\mathbb{S}^n}} \le |\p \hat \Omega^\delta |_{\varphi^* g_{\mathbb{S}^n}}\] and hence \[2 |\hat\Sigma \times \{l\} |_{\varphi^*g_{\mathbb{S}^n}} \le \liminf_{\delta\to0} |\hat \Omega^\delta|_{\varphi^*g_{\mathbb{S}^n}}.\] This proves that $ \lim_{\delta\to0} |\hat \Omega^\delta|_{\varphi^*g_{\mathbb{S}^n}}= P(\hat \Gamma)$.
Now, for any decreasing approximation by convex sets with non-empty interior $\hat\Gamma_k$, $\lim_{k\to\infty} |\Gamma_k|$ exists as it is a decreasing sequence. For each $k$, we may find $\delta_1(k)$ and $\delta_2(k)$ which converge to $0$ as $k\to \infty$ such that $\varphi(\hat \Omega^{\delta_1(k)}) \subset \hat \Gamma_k \subset \varphi(\hat \Omega^{\delta_2(k)})$ and this, in particular, implies that $|\hat \Omega^{\delta_1(k)}|_{\varphi^*g_{\mathbb{S}^n}} \le |\Gamma_k| \le |\hat \Omega^{\delta_2(k)}|_{\varphi^*g_{\mathbb{S}^n}}$. We conclude that  $|\Gamma_k| \to P(\hat \Gamma)$, as $k \to \infty$. 
 
  \end{proof}
 
 \end{lemma}

We conclude this appendix by  the following  generalized outer area minimizing property.
\begin{lemma}\label{lem-genouter} {\color{black}For $n\ge2$,} if  $\hat \Gamma_1, \hat \Gamma_2$,  are  compact convex sets such that $\hat \Gamma_1 \subset \hat \Gamma_2 \subset \mathbb{S}^n_+$. Then,  $P(\hat \Gamma_1)\le P(\hat \Gamma_2){\color{black}< |\mathbb{S}^{n-1}|}$. 

 \begin{proof}
 By Lemma \ref{lem-outersphere} and \ref{lem-b5}, there are approximating sequences of strictly decreasing compact convex sets $\hat \Gamma_{1,k}$ and $\hat \Gamma _{2,k}$ in $\mathbb{S}^{n}_+$ with smooth strictly convex boundaries. $|\Gamma_{1,k}| \to P(\hat\Gamma_1)$ and $|\Gamma_{2,k}|\to P(\hat\Gamma_2)$ by the two lemmas say. By the strict monotonicity of the sequences, for fixed $k$ there is $l_0$ such that $\hat \Gamma_{1,l} \subset \hat \Gamma_{2,k}$ for $l>l_0$.  Taking $l\to \infty$, the  outer area minimizing property in Lemma \ref{lem-outersphere} implies $P(\hat \Gamma_1) \le |\Gamma_{2,k}|$. The {\color{black}first} inequality now follows by letting  $k\to \infty$. {\color{black} The second inequality is implied by $\hat \Gamma_2 \subset \mathbb{S}^n \cap \{x_{n+1} \le \e\}$ for small $\e>0$, the first inequality, and  $|\mathbb{S}^n \cap \{x_{n+1}=\e\}| < |\mathbb{S}^{n-1}|$.}
 
 \end{proof}
\end{lemma}

\end{document}